\documentclass{birkjour}

\usepackage{amsfonts}
\usepackage{amsmath}
\usepackage{amssymb}
\usepackage{amsthm}

\usepackage{mathrsfs}
\usepackage{amscd}
\usepackage{amstext}

\newcommand{\re}{\mathbb{R}}\newcommand{\N}{\mathbb{N}}
\newcommand{\zz}{\mathbb{Z}}\newcommand{\C}{\mathbb{C}}

\newcommand{\Z}{{\zz}^n}

\newcommand{\R}{{\re}^n}

\newcommand{\cs}{{\mathcal S}}

\newcommand{\cl}{{\mathcal L}}

\newcommand{\cf}{{\mathcal F}}
\newcommand{\cfi}{{\cf}^{-1}}

\newcommand{\supp}{{\rm supp \, }}
\newcommand{\sign}{{\rm sign \, }}

\newcommand{\ls}{\lesssim}

\newcommand{\bproof}{\begin{proof}}
\newcommand{\eproof}{\end{proof}}
\renewcommand{\proofname}{\normalfont{\textbf{Proof}}}

\newcommand{\be}{\begin{equation}}
\newcommand{\ee}{\end{equation}}
\newcommand{\beq}{\begin{eqnarray}}
\newcommand{\beqq}{\begin{eqnarray*}}
\newcommand{\eeq}{\end{eqnarray}}
\newcommand{\eeqq}{\end{eqnarray*}}

 \newtheorem{thm}{Theorem}[section]
 \newtheorem{cor}[thm]{Corollary}
 \newtheorem{lem}[thm]{Lemma}
 \newtheorem{prop}[thm]{Proposition}
 \theoremstyle{definition}
 \newtheorem{defn}[thm]{Definition}
 \theoremstyle{remark}
 \newtheorem{rem}[thm]{Remark}
 
 \numberwithin{equation}{section}

\begin{document}

\title{Multiplication and Composition in  Weighted Modulation Spaces}

\author[Reich]{Maximilian Reich}
\address{%
TU Bergakademie Freiberg\\
Institut f\"ur Angewandte Analysis\\
09596 Freiberg\\
Germany}
\email{maximilian.reich@math.tu-freiberg.de}
\author[Sickel]{Winfried Sickel}
\address{%
Friedrich-Schiller-University Jena \\
Ernst-Abbe-Platz 2 \\
07737 Jena \\
Germany}
\email{winfried.sickel@uni-jena.de}



\date{today}

\begin{abstract}
We study the existence of the product of two weighted modulation spaces.
For this purpose we discuss two different strategies.
The more simple one allows transparent proofs in various situations.
However, our second method allows a closer look onto  associated 
norm inequalities under restrictions in the Fourier image.
This will give us the opportunity to treat the boundedness of composition operators. 
\end{abstract}

\maketitle


\section{Introduction} \label{introduction}


Since modulation spaces have been introduced by Feichtinger \cite{feichtinger}
they have become canonical for both time-frequency and phase-space analysis.
However, in recent time modulation spaces have been found useful also in connection with linear and nonlinear partial differential equations, see, e.g.,
Wang et all \cite{wangExp,wang,wang1,wang2}, Ruzhansky, Sugimoto and Wang \cite{ruzhansky} or  Bourdaud, Reissig, S. \cite{brs}.
Investigations of partial differential equations require partly different tools than used in  time-frequency and phase-space analysis.
In particular, Fourier multipliers,  pointwise multiplication and composition of functions need to be studied.
In our contribution we will concentrate on pointwise multiplication and composition of functions.
Already Feichtinger \cite{feichtinger} was aware of the importance of pointwise multiplication in modulation spaces. 
In the meanwhile several authors have studied this problem, we refer, e.g., to \cite{cordero},  \cite{iwabuchi}, \cite{Sugi} and \cite{toftCont}, 
\cite{toftAlg}.
In Section \ref{mult} we will give a survey about the known results. Therefore we will discuss two different proof strategies.
The more simple one, due to  Toft \cite{toftCont, toftAlg} and Sugimoto, Tomita and  Wang  \cite{Sugi},   
allows transparent proofs in various situations, in particular one can deal with those situations where the modulation spaces 
form algebras with respect to pointwise multiplication.
As a consequence, Sugimoto et all  \cite{Sugi} are able to deal with composition operators on modulation spaces induced by analytic functions. 
Our second method, much more complicated, allows a closer look onto  associated 
norm inequalities under restrictions in the Fourier image.
This will give us the possibility  to discuss  the boundedness of composition operators on weighted modulation spaces based on a technique which goes back to Bourdaud \cite{Bou},
see also Bourdaud, Reissig, S. \cite{brs} and Reich, Reissig, S. \cite{rrs}.
Our approach will allow to deal with the boundedness of nonlinear operators $T_f:~g \mapsto f \circ g $ without assuming $f$ to be analytic.
However, as the case of $M^s_{2,2}$ shows, our sufficient conditions are not very close to the necessary conditions. There is still a certain gap.
\\
The paper is organized as follows.
In Section \ref{basics} we collect what is needed about the weighted modulation spaces we are interested in.
The next section  is devoted to the study of pointwise multiplication. In particular, we are interested in embeddings 
of the type
\[ M^{s_1}_{p,q} \, \cdot \, M^{s_2}_{p,q} \hookrightarrow M^{s_0}_{p,q}\, , \]
where $s_1,s_2, p$ and $q$ are given and we are asking for an optimal $s_0$.
These results will be applied to problems around the regularity of composition of functions  in Section \ref{comp}. 
For convenience of the reader we also recall what is known in the more general 
situation 
\[ M^{s_1}_{p_1,q_1} \, \cdot \, M^{s_2}_{p_2,q_2} \hookrightarrow M^{s_0}_{p,q}\, . \]
Special attention will be  paid to the algebra property.
Here the known sufficient conditions are supplemented by necessary conditions, see Theorem \ref{mult2}.
Also only partly new is our main result in Section \ref{mult} stated in Theorem \ref{algebra12}.
Here we investigate multiplication of distributions (possibly singular) with regular functions (which are not assumed to be $C^\infty$).
Partly we have found necessary and sufficient conditions also in this more general situation.
Finally, Section \ref{comp} deals with composition operators.
As direct consequences of the obtained results for pointwise multiplication 
we can deal with the mappings $g \mapsto g^\ell$, $\ell\ge 2$, see Subsection \ref{poly}.
In Subsection \ref{comp2}
we shall investigate $g \mapsto f \circ g$, where $f$ is not assumed to be analytic.
Sufficient conditions, either in terms of a decay for $\cf f$ or in terms of regularity of $f$, are given.


\subsection*{Notation} 


We introduce some basic notation. 
As usual, $\N$ denotes the natural numbers, $\N_0 := \N \cup \{0\}$,
$\zz$ the integers and
$\re$ the real numbers, $\C$ refers to the complex numbers. For a real number $a$ we put $a_+ := \max(a,0)$.
For $x \in \R$ we use $\|x\|_\infty:= \max_{j=1, \ldots \,,n} \, |x_j|  $. 
Many times we shall use the abbreviation  $\langle \xi\rangle := (1+|\xi|^2)^{\frac{1}{2}}$, $\xi \in \R$.
\\
The symbols  $c,c_1, c_2, \, \ldots  \, ,C, C_1,C_2, \,  \ldots $ denote   positive
constants which are independent of the main parameters involved but
whose values may differ from line to line. 
The notation $a \lesssim b$ is equivalent to $a \leq Cb$ with a positive constant $C$. 
Moreover, by writing $a \asymp b$ we mean $a \lesssim b \lesssim a$.
\\
Let $X$ and $Y$ be  two Banach spaces. Then the symbol $X \hookrightarrow Y$ indicates that the embedding is continuous.
By $\cl (X,Y)$ we denote the collection of all linear and continuous operators which map $X$ into $Y$.
By $C_0^\infty (\R)$ the set of compactly supported infinitely differentiable functions $f:\R \to \C$ is denoted.
Let $\cs (\R)$ be the Schwartz space of all complex-valued rapidly decreasing infinitely differentiable  functions on $\R$. 
The topological dual, the class of tempered distributions, is denoted by $\cs'(\R)$ (equipped with the weak topology).
The Fourier transform on $\cs(\R)$ is given by 
\[
\cf \varphi (\xi) = (2\pi)^{-n/2} \int_{\R} \, e^{ix \cdot \xi}\, \varphi (x)\, dx \, , \qquad \xi \in \R\, .
\]
The inverse transformation is denoted by $\cfi $.
We use both notations also for the transformations defined on $\cs'(\R)$.\\
{\bf Convention.} If not otherwise stated all functions will be considered on the Euclidean $n$-space $\R$.
Therefore $\R$ will be omitted in notation.


\section{Basics on Modulation Spaces} \label{basics}



\subsection{Definitions} \label{Definitions}


A general reference for definition and properties of  weighted modulation spaces  
is Gr\"ochenig's monograph \cite[Chapt.~11]{groechenig}.

\begin{defn} \label{STFT}
Let $\phi \in  \cs $ be nontrivial. 
Then the short-time Fourier transform  of a function $f$ with respect to $\phi$ is defined as
\[ V_\phi f(x,\xi) = (2\pi)^{-\frac{n}{2}} \int_{\R} f(s) \overline{\phi(s-x)} e^{-i s \cdot \xi} \, ds 
\qquad (x,\xi \in \R). 
\]
\end{defn}

The function $\phi$ is usually called the window function. For $f \in \cs'$ 
the short-time Fourier transform $V_\phi f$ is a continuous function of at most polynomial growth on $\re^{2n}$, see \cite[Thm.~11.2.3]{groechenig}. 

\begin{defn} \label{modCont}
Let $1\leq p,q\leq\infty$. Let $\phi\in \cs$ be a 
fixed window and assume $s \in\re$.  Then the weighted modulation 
space ${M}_{p,q}^{s}$ is the collection  of all $f \in \cs'$ such that 
\[
\|f\|_{{M}_{p,q}^{s}} = \Big( \int_{\R} 
\Big( \int_{\R} |V_\phi f(x,\xi) \,  \langle \xi\rangle^s|^p dx \Big)^{\frac{q}{p}} d\xi \Big)^{\frac{1}{q}} < \infty\,  
\]
(with obvious modifications if $p= \infty$ and/or $q=\infty$).
\end{defn}

Formally these spaces ${M}_{p,q}^{s}$ depend on the window $\phi$. However, for different windows $\phi_1, \phi_2$ the resulting spaces 
coincide as sets and the norms are equivalent, see \cite[Prop.~11.3.2]{groechenig}.
For that reason we do not indicate the window in the notation (we do not distinguish spaces which differ only by an equivalent norm).

\begin{rem}\label{hs=ms}
\rm
(i) General references with respect to weighted modulation spaces are 
Feichtinger \cite{feichtinger}, Gr\"ochenig \cite[Chapt.~11]{groechenig}, Gol'dman \cite{go},
Guo et all \cite{guo},  Toft \cite{toftCont}, \cite{toftConv}, \cite{toftAlg}, Triebel \cite{trzaa} and
Wang et. all \cite{wangExp} to mention only a few.
\\
(ii) There is an important special case. In case of $p=q=2$ we obtain $M^s_{2,2} = H^s$ in the sense of equivalent norms, see
Feichtinger \cite{feichtinger}, Gr\"ochenig \cite[Prop.~11.3.1]{groechenig}.
Here $H^s$ is nothing but the standard Sobolev space built on $L_2$, at least for $s\in \N$. In general $H^s$ is the collection of all $f \in \cs'$ such that
\[
 \| f\|_{H^s} := \Big(\int_{\R} (1+|\xi|^2)^{s}\, |\cf f (\xi)|^2 \, d\xi \Big)^{1/2}<\infty.
\]
\end{rem}

For us of great use will be another  alternative approach to the spaces $M^s_{p,q}$.  
This will be more close to the standard techniques used in connection with Besov spaces.
We shall use the so-called frequency-uniform 
decomposition, see , e.g.,  Wang \cite{wang}. 
Therefore, let $\rho: \R \mapsto [0,1]$ be a Schwartz function which 
is compactly supported in the cube 
\[
Q_0 := \{ \xi \in \R: -1\leq \xi_i \leq 1,\:  i=1,\ldots,n \}\, .
\] 
Moreover, we assume 
\[
\rho(\xi)=1 \qquad \mbox{if}  \quad  |\xi_i|\leq \frac{1}{2}\, , \qquad i=1, 2, \ldots , n. 
\]
With $\rho_k (\xi) :=\rho (\xi-k) $, $\xi \in \R$, $k \in \Z$, it follows
\[
\sum_{k \in \Z} \rho_k (\xi)\ge 1 \qquad \mbox{for all}\quad \xi \in \R\, .
\]
Finally we define
\[ 
\sigma_k(\xi) := \rho_k(\xi) \Big(\sum_{k\in\Z} \rho_k(\xi)\Big)^{-1}, \qquad \xi \in \R\, , \quad k\in\Z \, .
\]
The following  properties are obvious: 
\begin{itemize}
			\item $ 0 \le \sigma_k(\xi) \le 1$ for all $\xi \in \R$;
			\item $\supp \sigma_k \subset Q_k := \{ \xi \in \R: -1\leq \xi_i -k_i \leq 1, \: i=1,\ldots,n \} $;
			\item $\displaystyle \sum_{k\in\Z} \sigma_k(\xi) \equiv 1$ for all $\xi\in\R$;
			\item There exists a constant $C>0$ such that $\sigma_k (\xi) \ge C$ if $ \max_{i=1, \ldots , n}\, |\xi_i-k_i|\leq \frac{1}{2}$; 
\item For all $m \in \N_0$ there exist positive constants $C_m$ such that for $|\alpha|\leq m$
\[
\sup_{k \in \Z}\, \sup_{\xi \in \R} \, |D^\alpha \sigma_k(\xi)|\leq C_m\,   .
\]
\end{itemize}
We shall call the mapping
\[ \Box_k f := \cfi \left[ \sigma_k (\xi) \, \cf f(\xi)\right] (\cdot ), \qquad k\in\Z, \quad f \in \cs'\, , 
\]
frequency-uniform decomposition operator. 

As it is well-known there is an equivalent description of the 
modulation spaces by means of the frequency-uniform decomposition operators.

\begin{prop} \label{defdecomp}
Let $1\leq p,q \leq \infty$ and assume $s \in\re$. 
Then the weighted modulation space $M_{p,q}^{s}$ consists of all tempered distributions $f\in \cs'$ such that 
\[ \|f\|^*_{M_{p,q}^{s}} = \Big( \sum_{k\in\Z} \langle k \rangle^{sq} \|\Box_k f\|_{L^p}^q \Big)^{\frac{1}{q}} < \infty \, . 
\]
Furthermore, the norms $ \|f\|_{M_{p,q}^{s}}$ and  $\|f\|^*_{M_{p,q}^{s}}$ are equivalent.
\end{prop}

We refer to  Feichtinger \cite{feichtinger} or Wang and  Hudzik  \cite{wang}. 
In what follows we shall  work with both characterizations. In general we shall use the same notation $\|\, \cdot \, \|_{M_{p,q}^{s}}$ for both norms.

\begin{lem}\label{basic}
(i) The modulation space ${M}_{p,q}^{s}$ is a Banach space.
\\
(ii) ${M}_{p,q}^{s}$ is independent of the choice of the window $\rho\in C_0^\infty $ in the sense of equivalent norms.
\\
(iii) $M^s_{p,q}$ is continuously embedded into $\cs'$.
\\
(iv) ${M}_{p,q}^{s}$ has the Fatou property, i.e., if $(f_m)_{m=1}^\infty \subset {M}_{p,q}^{s}$ 
is a sequence such that 
$ f_m \rightharpoonup f $ (weak convergence in $\cs'$) and 
\[
 \sup_{m \in \N}\,  \|\, f_m \, \|_{{M}_{p,q}^{s}} < \infty\, , 
\]
then $f \in {M}_{p,q}^{s}$ follows and 
\[
\|\, f \, \|_{{M}_{p,q}^{s}} \le \sup_{m\in \N}\,  \|\, f_m \, \|_{{M}_{p,q}^{s}} < \infty\, .
\]
\end{lem}

\begin{proof}
For (i), (ii), (iii) we refer to \cite{groechenig}.\\
We comment on a proof of (iv). Therefore we follow \cite{fr} and work 
with the norm $\|\, \cdot \, \|^*_{M_{p,q}^{s}}$.  From the assumption we obtain  that for all $k\in \Z$ and $x\in\R$,
\begin{eqnarray*}
\cfi \, [\sigma_k \, \cf f_m](x)=  (2 \pi)^{-n/2} \, f_m(x-\cdot)(\sigma_k) \to f(x-\cdot)(\sigma_k)= 
\cfi \, [\sigma_k \, \cf f](x) 
\end{eqnarray*}
as $m\to\infty$. Fatou's lemma yields
\begin{eqnarray*}
\sum_{|k| \le N}  \Big(\int_{\R} \, && \hspace{-0.7cm} |\cfi  \, [\sigma_k \, \cf f] (x)|^p dx\, \Big)^{\frac qp} 
\\
&\le & \liminf_{m\to\infty} \sum_{|k|\le N} \, \Big(\int_{\R} \,  |\cfi \, [\sigma_k \, \cf f_m](x)|^p dx \, \Big)^{\frac qp} \, .
\end{eqnarray*}
An obvious monotonicity argument completes the proof.
\end{proof}


\subsection{Embeddings}


Obviously the  spaces ${M}_{p,q}^{s}$ are monotone in $s$ and $q$. But they are also monotone with respect to $p$. 
To show this we recall Nikol'skij's inequality, see, e.g.,  Nikol'skij \cite[3.4]{Ni} or Triebel \cite[1.3.2]{triebel}.

\begin{lem} \label{nikolskij}
Let $1\leq p \leq q \leq\infty$ and $f$ be an integrable function with $\supp \cf f \subset B(y,r)$, i.e., 
the support of the Fourier transform of $f$ is contained in a ball with radius $r>0$ and center in $y \in \R$. Then it holds
		\[ \|f\|_{L_q} \leq C r^{n(\frac{1}{p}-\frac{1}{q})} \|f\|_{L_p} \]
with a constant $C>0$ independent of $r$ and $y$.
\end{lem}

This implies $\|\Box_k f\|_{L_q} \le c \, \|\Box_k f\|_{L_p}$ if $p \le q$ with $c$ independent of $k$ and $f$ which results in the following
corollary (by using the norm $\|\, \cdot \, \|^*_{M_{p,q}^{s}}$).

\begin{cor}\label{einbettung}
Let $s_0 > s$, $p_0 < p $ and $q_0 < q$.
Then the following  embeddings hold and are continuous:
\[
{M}_{p,q}^{s_0} \hookrightarrow {M}_{p,q}^{s}\, , \qquad M_{p_0,q}^{s} \hookrightarrow {M}_{p,q}^{s} 
\]
and
\[
{M}_{p,q_0}^{s} \hookrightarrow {M}_{p,q}^{s}\, ; 
\]
i.e., for all $p,q$, $1\le p,q\le \infty$, we have 
\[
{M}_{1,1}^{s} \hookrightarrow 
{M}_{p,q}^{s} \hookrightarrow  {M}_{\infty,\infty}^{s}\, .
\]
\end{cor}

Of some importance are embeddings with respect to different metrics.
To find sufficient conditions is not difficult when working with  $\|\, \cdot \, \|^*_{M_{p,q}^{s}}$.
A bit more tricky are the necessity parts.
We refer to the recent paper by Guo et all \cite{guo}.

\begin{prop}\label{einbettung1}
Let $s_0, s_1 \in \re$ and  $1 \le p_0, p_1 \le \infty$.
Then 
\[
{M}_{p_0,q_0}^{s_0} \hookrightarrow {M}_{p_1,q_1}^{s_1}
\]
holds if and only if either
\begin{itemize}
 \item  $p_0 \le  p_1$ and  $s_0 - s_1 > n \Big(\frac{1}{q_1} - \frac{1}{q_0}\Big)$  
 \item or $p_0 \le  p_1$, $s_0 = s_1$ and  $ q_0 = q_1$.  
\end{itemize}
\end{prop}

\begin{rem}
\rm
Embeddings of modulation spaces are treated at various places, we refer to Feichtinger \cite{feichtinger}, 
Wang, Hudzik \cite{wang}, Cordero, Nicola \cite{cordero}, Iwabuchi \cite{iwabuchi} and Guo, Fan, Wu and Zhao \cite{guo}. 
\end{rem}

The weighted modulation spaces $M^s_{p,q}$ cannot distinguish between boundedness and continuity (as Besov spaces).
Let 
$C_{ub}$ denote the class of all uniformly continuous and bounded functions $f:~\R \to \C$ equipped with the supremum norm.
If $f \in M^{s}_{p,q} $ is a regular distribution it is determined (as a function) almost everywhere.
We shall say that $f$ is a continuous function if there is one continuous function $g$ which 
equals $f$ almost everywhere.

\begin{cor}\label{einbettung2}
Let $s \in \re$ and $1 \le p,q\le \infty $.
Then the following assertions are equivalent:
\begin{itemize}
 \item ${M}_{p,q}^{s} \hookrightarrow L_\infty  $;
 \item ${M}_{p,q}^{s} \hookrightarrow C_{ub}  $;
 \item ${M}_{p,q}^{s} \hookrightarrow M^0_{\infty,1}  $;
 \item either $s\ge 0$ and $q=1$ or $s>n/q'$.
\end{itemize}
\end{cor}

\begin{proof}
We shall work with $\| \, \cdot \, \|^*_{M_{p,q}^{s}}$. \\
Step 1. Sufficiency. By Proposition  \ref{einbettung1} it will be enough to show $M^0_{\infty,1} \hookrightarrow C_{ub}$. From the definition of 
$M^0_{\infty,1}$ it follows that
\[
\sum_{k \in \Z} \Box_k f (x) 
\]
is pointwise convergent (for all $x\in \R$). Furthermore, since $\Box_k f \in C^\infty$, 
there is a continuous representative in the equivalence class $f$, given by $\sum_{k \in \Z} \Box_k f (x)$.
In what follows we shall work with this representative. Boundedness of $f \in M^0_{\infty,1}$ is obvious,
we have 
\[ |f(x)| = |\sum_{k\in \Z} \Box_k f(x)| \le \| \, f \,\|_{M^0_{\infty,1}} \, . \]
It remains to prove uniform continuity.
For fixed $\varepsilon >0 $ we choose $N$ such that 
\[
\sum_{|k|>N} \| \, \Box_k f\, \|_{L_\infty } < \varepsilon /2\, .
\]
In case $|k|\le N $ we observe that 
\[
|\Box_k f(x)- \Box_k f(y)| \le \|\, \nabla  (\Box_k f)  \, \|_{L_\infty} |x-y|\, .
\]
It follows from \cite[Thm.~1.3.1]{triebel} that
\[
\|\, \nabla  (\Box_k f)  \, \|_{L_\infty }\le c_1 \, \| \, (M \Box_k f) \, \|_{L_\infty }
\]
with a constant $c_1$ independent of $f$ and $ k$. Here $M$ denotes the Hardy-Littlewood maximal function. 
In the quoted reference the assumption $\Box_k f \in \cs $ is used.
A closer look at the proof shows that 
$\Box_k f \in L_1^{\ell oc} $ satisfying
\[
\int_{Q_k} |\Box_k f(x)|\, dx \le c_2\, (1+|k|)^N\, , \qquad k \in \Z\, , 
\]
for some $N \in \N$ is sufficient.
Since $\Box_k f \in L_\infty $ this is obvious.
Consequently we obtain
\beqq
|\Box_k f(x)- \Box_k f(y)| & \le &  c_1 \, \| \, (M \Box_k f) \, \|_{L_\infty }\, |x-y| 
\le c_1\, \|\, \Box_k f\, \|_{L_\infty}  \, |x-y|
\\
& \le &  c_2\, \|\,  f\, \|_{L_\infty}   \, |x-y|\, , 
\eeqq
where in the last step we used the standard convolution inequality \\
$\|\, g * h \, \|_{L_\infty}  
\le \, \|\,  g\, \|_{L_1} \|\,  f\, \|_{L_\infty} $.
This implies uniform continuity of $\Box_k f$ and therefore of $\sum_{|k|\le N} \Box_k f$. 
In particular, we find
\beqq
|f(x)-f(y)| & = & \Big| \sum_{k\in \Z } (\Box_k f (x) - \Box_k f (y))  \Big| 
\\
& \le &  \sum_{|k| > N} (|\Box_k f (x)|  + | \Box_k f (y)|) + c_2\, \|\, f \, \|_{L_\infty} \, |x-y| \, \sum_{|k| \le  N} 1    
\\
& \le &  \varepsilon  +  c_2\, \|\, f \, \|_{L_\infty} \, |x-y| \, (2N+1)^n   \, .
\eeqq
Choosing $\delta = (c_2 \,  \|\, f \, \|_{L_\infty}\,  (2N+ 1)^{n})^{-1} \, \varepsilon$
we arrive at
\[
 |f(x)-f(y)| < 2 \, \varepsilon \qquad \mbox{if}\qquad |x-y|< \delta \, .
\]
{\em Step 2.} Necessity.
Let $\psi \in \cs $ be a real-valued function such that $\psi (0)=1$ and 
\[
\supp \cf \psi \subset \{\xi: \max_{j=1, \ldots , n}\, |\xi_j|< \varepsilon\}\,  \qquad \mbox{with}\quad  \varepsilon < 1/2.
\]
We define $f$ by 
\[
\cf f (\xi) := \sum_{k \in \Z} a_k \, \cf \psi (\xi -k)\, . 
\]
Clearly, 
\[
 \Box_k f (x) = a_k \, e^{ikx}\,  \psi (x)\, , \qquad k \in \Z\, .
\]
{\em Substep 2.1.} Let $s=0$ and $1 \le p \le  \infty$.
The above arguments  imply 
$f \in M^0_{p,q}$ if and only if $(a_k)_k \in \ell_q$.
On the other hand,
\be\label{ws-10b}
 f (x) = \psi (x)\, \sum_{k \in \Z} a_k \, e^{ikx}
\ee
which implies that $f $ is unbounded in $0$ if $\sum_{k\in\Z} a_k = \infty$.
Choosing 

\[a_k := \left\{\begin{array}{lll} 
(k_1 \, \log (2+ k_1))^{-1}\, &\qquad & \mbox{if}\quad k_1 \in \N\, , \quad  k=(k_1, 0, \ldots \, 0)\,;
\\
0 && \mbox{otherwise};
\end{array}\right. 
\]
then $f \in M^0_{p,q} \setminus L_\infty  $, $q >1 $, follows.
\\
{\em Substep 2.2.} Let $1 \le p \leq \infty$ and $q=\infty$.
Then we choose $a_k := \langle k\rangle^{-n} $. It follows 
$f \in M^n_{p,\infty}$ but $f(0)=+\infty$.
\\
{\em Substep 2.3.} Let $1\le p \leq \infty$, $1 <q<\infty$ and $s=n/q'$.
Then, with $\delta >0$, we choose

\[a_k := \left\{\begin{array}{lll} 
\langle k \rangle^{-n}  \, \log \langle k \rangle^{-(1+\delta)/q} \, &\qquad & \mbox{if}\quad |k| >0\,;
\\
0 && \mbox{otherwise}.
\end{array}\right. 
\]
It follows
\begin{eqnarray*}
	\| \, f \, \|_{M^{n/q'}_{p,q}} & = &  \|\psi\|_{L_p}\, \sum_{|k|>0} \langle k \rangle^{-nq+ nq/q'}  \, (\log \langle k \rangle)^{-(1+\delta)}  \\
	& = &  \|\psi\|_{L_p}\, \sum_{|k|>0} \langle k \rangle^{-n}  \, (\log \langle k \rangle)^{-(1+\delta)} <\infty\, .
\end{eqnarray*}
On the other hand we have
\[
 f(0) = \sum_{|k|>0} \langle k \rangle^{-n}  \, \log \langle k \rangle^{-(1+\delta)/q}= \infty
\]
if $(1+\delta)/q \le  1$. Hence, for choosing $\delta = q-1$ the claim follows.
\end{proof}

\begin{rem}
 \rm
 Sufficient conditions for embeddings of modulation spaces  into spaces of continuous functions can be found at several places,
 in particular in Feichtinger's original paper \cite{feichtinger}. We did not find references for the necessity.
\end{rem}


\section{Pointwise Multiplication in Modulation Spaces} \label{mult}


We are interested in embeddings  of the type
\[ M^{s_1}_{p,q} \, \cdot \, M^{s_2}_{p,q} \hookrightarrow M^{s_0}_{p,q}\, , \]
where $s_1,s_2, p$ and $q$ are given and we are asking for an optimal $s_0$.
These results will be applied in connection with our investigations on the regularity of compositions of functions  in Section \ref{comp}. 
However, several times we shall deal with the slightly more general problem 
\[ M^{s_1}_{p_1,q} \, \cdot \, M^{s_2}_{p_2,q} \hookrightarrow M^{s_0}_{p,q}\, , \qquad 
\frac 1p = \frac{1}{p_1} + \frac{1}{p_2}\, .
\]
In view of Corollary \ref{einbettung} this always yields
\[ M^{s_1}_{p_1,q} \, \cdot \, M^{s_2}_{p_2,q} \hookrightarrow M^{s_0}_{p,q}\, , \qquad 
\frac 1p \le  \frac{1}{p_1} + \frac{1}{p_2}\, .
\]
For convenience of the reader we also recall what is known in the more general 
situation 
\[ M^{s_1}_{p_1,q_1} \, \cdot \, M^{s_2}_{p_2,q_2} \hookrightarrow M^{s_0}_{p,q}\, . \]
At first we shall deal with the algebra property. Afterwards we  turn to the existence of 
the product in more general situations.


\subsection{On the algebra property} \label{known}


The main aim consists in giving necessary and sufficient conditions for the 
embedding $ M^{s}_{p,q} \, \cdot \, M^{s}_{p,q} \hookrightarrow M^{s}_{p,q}$.
To prepare this we recall a nice identity due to Toft \cite{toftCont}, see also 
Sugimoto, Tomita and Wang \cite{Sugi}.

\begin{lem}\label{lemma1}
Let $\varphi_1, \varphi_2 \in \cs $ be nontrivial. Let $f,g \in L_2^{\ell oc} $ such that there exist $c>0$ and $M>0$ with 
\[
\int_{Q_k} |f(x)|^2 + |g(x)|^2 \, dx \le c \, (1+|k|)^{M} \, , \qquad k \in \Z\, .
\]
For all $x, \xi \in \R$ the following identity takes place
\be\label{ws-1}
 V_{\varphi_1 \, \cdot \, \varphi_2} (fg) (x,\xi) = (2\pi)^{-n/2} \int V_{\varphi_1} (f) (x,\xi-\eta) \, V_{\varphi_2} (g) (x,\eta)\, d\eta \, . 
\ee
\end{lem}

\begin{proof}
The main tool will be the Plancherel identity. 
Observe, that for any fixed $x \in  \R$ the functions $f(t)\,\overline{\varphi_1 (t-x)}$, $\overline{g(t)}\,{\varphi_2 (t-x)}$ belong to $L_2 $ and
therefore their Fourier transforms as well. For brevity we put
\[I :=
\int V_{\varphi_1} (f) (x,\xi-\eta) \, V_{\varphi_2} (g) (x,\eta)\, d\eta \, . 
\]
Applying the Plancherel identity  we conclude
\beqq
I&=&   \int   \cf (  f(t)\,\overline{\varphi_1 (t-x)} \, e^{-i\xi  t})(-\eta)\,  
\overline{\cf (\overline{g(t)}\,{\varphi_2 (t-x)})(-\eta)} \, d\eta
\\
&=&   \, \int    f(t)\,\overline{\varphi_1 (t-x)} \, e^{-i\xi t} \,  
\overline{\overline{g(t)}\,{\varphi_2 (t-x)})} \, dt
\\
& = & \int f(t)\,g(t)\,  \overline{\varphi_1 (t-x) \, \varphi_2 (t-x)}\, e^{-i\xi  t}\, dt 
\\
& = &  (2\pi)^{n/2}\, V_{\varphi_1 \cdot \varphi_2} (fg) (x,\xi)
\, .
\eeqq
The proof is complete.
\end{proof}

\begin{rem}
\rm
 It is clear that the assertion does not extend very much.
 E.g., if  $f,g \in L_p^{\ell oc} $ for some $p< 2$ then the above claim is not true.
 We may take 
 \[
 f(x) = g(x)= \psi (x) \, |x|^{-n/2}\, , \qquad x \in \R\, ,
 \]
 where $\psi$ is a smooth and compactly supported cut-off function s.t. $\psi (0)=1$.
 Then $f \, \cdot \, g $ is not longer a distribution, i.e., the integral 
 \[
  V_{\varphi_1 \,\cdot \,  \varphi_2} (fg)(x,\xi) = (2\pi)^{-\frac{n}{2}} \int_{\R} f(s)\, g(s)\,  \overline{\varphi_1(s-x) \, \varphi_2 (s-x)}\,  e^{-i s\cdot \xi} \, ds 
 \]
 does not make sense in general.
\end{rem}

In \cite{toftCont} and \cite{Sugi} the identity \eqref{ws-1} is applied either in case $f,g \in \cs$
or $f,g \in L_\infty$. Here we shall apply it in the wider context of Lemma \ref{lemma1}.

\begin{lem}\label{mult1}
Let $1\le p,q\le \infty$ and assume $M^s_{p,q} \hookrightarrow M^0_{\infty,1}$. 
Then there exists a constant $c$ such that
\[
\| \, f \, \cdot \, g \, \|_{M^s_{p,q}} \le c\,  \Big(\| \, f \|_{M^0_{\infty,1}} \, \| \, g \|_{M^s_{p,q}}  +  \| \, f \|_{M^s_{p,q}} \, \| \, g \|_{M^0_{\infty,1}} \Big)
\]
holds for all $f,g \in M^s_{p,q}$.
\end{lem}

\begin{proof}
The main idea in the proof consists in the fact that the modulation space can be characterized by different window functions.  
Since $M^0_{\infty,1} \hookrightarrow L_\infty$ we know that $f,g$ satisfy the conditions in Lemma \ref{lemma1}. Hence
\beq\label{ws-14B}
\|\, f\, \cdot \, g\, \|_{M^s_{p,q}} & = &  \Big\{ \int \Big[ \int \Big|  V_{\varphi^2} (f\, \cdot \, g)(x,\xi)\, 
\Big|^p dx\Big]^{q/p} \langle \xi \rangle^{sq} d\xi \Big\}^{1/q} \\
&\leq& \Big\{ \int \Big[ \int \Big|  \int V_{\varphi} f(x,\xi-\eta) V_\varphi g (x,\eta)  d\eta\, 
\Big|^p dx\Big]^{q/p} \langle \xi \rangle^{sq} d\xi \Big\}^{1/q}.
\nonumber
\eeq
We split the integration with respect to $\eta$ into two parts
\be\label{ws-15B}
\Omega_\xi := \{ \eta: \: |\xi - \eta|\ge |\eta|\} \qquad \mbox{and}\qquad \Gamma_\xi := \{\eta: \: |\xi - \eta|< |\eta|\}\, , \quad \xi \in \R\, .  
\ee
It follows
\[
 \|\, f\, \cdot \, g\, \|_{M^s_{p,q}} \le 2^s \, (A+ B)\, , 
\]
where 
\[
 A:= \Big\{ \int \Big[ \int \Big|  \int_{\Omega_\xi} V_{\varphi} f(x,\xi-\eta)\, (1+|\xi-\eta|^2)^{s/2} V_\varphi g (x,\eta)\,  d\eta\, \Big|^p dx\Big]^{q/p} \, d\xi \Big\}^{1/q}
\]
and 
\[
 B:= \Big\{ \int \Big[ \int \Big|  \int_{\Gamma_\xi} V_{\varphi} f(x,\xi-\eta)\, V_\varphi g (x,\eta)\, (1+|\eta|^2)^{s/2} d\eta\, \Big|^p dx\Big]^{q/p} \, d\xi \Big\}^{1/q}\, .
\]
We continue by applying the generalized Minkowski inequality, see \cite[Thm.~2.4]{LL}. This yields
\beqq
 A &\le & \int \Big\{ \int \Big[ \int | V_{\varphi} f(x,\xi-\eta)\, \langle \xi-\eta \rangle^{s} \, V_\varphi g (x,\eta)\,|^p \,  dx\Big]^{q/p} \, d\xi \Big\}^{1/q}\, d\eta
 \\
 & \le & \int \Big(\sup_{x \in \R} \, | V_\varphi g (x,\eta)|\Big)\, \Big\{ \int 
 \Big[ \int | V_{\varphi} f(x,\xi-\eta)\, \langle \xi-\eta \rangle^{s} \,|^p \,  dx\Big]^{q/p} \, d\xi \Big\}^{1/q}\, d\eta
 \\
 &=&  \|\, g\, \|_{M^0_{\infty,1}}  \, \|\, f\, \|_{M^s_{p,q}} \, .
\eeqq
Analogously one can prove 
\[
 B \le  \|\, f\, \|_{M^0_{\infty,1}}  \, \|\, g\,  \|_{M^s_{p,q}} \, .
\]
The proof is complete.
\end{proof}

\begin{rem}
 \rm
 (i)
We proved a bit more than stated.
In fact, we have shown
\[
\| \, f \, \cdot \, g \, \|_{M^s_{p,q}} \le 2^{s}\,  \Big(\| \, f \|_{M^0_{\infty,1}} \, \| \, g \|_{M^s_{p,q}}  +  \| \, f \|_{M^s_{p,q}} \, \| \, g \|_{M^0_{\infty,1}} \Big)
\]
But here one has to notice that the norm on the left-hand side is generated by the window $\varphi^2$, whereas the norms on the right-hand side are generated by the window $\varphi$.
\\
(ii) Lemma \ref{mult1} has been proved by  Sugimoto, Tomita and Wang \cite{Sugi}. For partial results with a different proof we refer to Feichtinger \cite{feichtinger}.
\end{rem}

Next we turn to necessary and sufficient conditions for the algebra property.

\begin{thm}\label{mult2}
Let $1\le p,q\le \infty$ and $s \in \re $. 
Then $M^s_{p.q}$ is an algebra with respect to pointwise multiplication if and only if 
either $s \ge 0 $ and $q=1$ or $s>n/q'$.
\end{thm}

\begin{rem}\label{mult3}
 \rm
(i) By Corollary \ref{einbettung2} the Theorem \ref{mult2} can be reformulated as 
 \[
 M^s_{p.q}\quad  \mbox{is an algebra} \qquad \Longleftrightarrow \qquad  
 M^s_{p.q} \hookrightarrow L_\infty\, . 
 \]
This is in some sense natural because otherwise one could increase local singularities by pointwise multiplication.
\\
(ii)
Theorem \ref{mult2} has a partial counterpart for Besov spaces.
Here one knows that $B^s_{p,q}$ is an algebra if and only if $B^s_{p,q} \hookrightarrow L_\infty$ and $s>0$.
We refer to Peetre \cite[Thm.~11, page 147]{peetre}, Triebel \cite[Thm.~2.8.3]{triebel} (sufficiency) and to \cite[Thm.~4.6.4/1]{RS} (necessity).
\end{rem}

To prepare the proof we need the following lemma which is of interest for its own.

\begin{lem}\label{notw}
Let  $1\le p,q < \infty$ and $s \in \re $.
Let $f \in \cs'$ and let  there exists a constant $c>0$ such that
\[
\|f \, \cdot \, g\|_{M^s_{p,q}} \le c \, \|g\|_{M^s_{p,q}}
\]
holds for all $g \in \cs$. Then $f\in L_\infty$ follows.
\end{lem}

\begin{proof}
 Let $T_f (g):= f \, \cdot \, g$, $g\in \cs$.
Let $\mathring{M}^s_{p,q}$ denote the closure of $\cs$  in $M^s_{p,q}$.
Hence, there is a unique extension of $T_f$ to a continuous operator belonging to 
$\cl (\mathring{M}^s_{p,q}, {M}^s_{p,q})$.
Next we employ duality. We fix $p,q$ and $s$ ($1\le p,q\le \infty$, $s\in \re$).
Let $(g,h)$ denote the standard dual pairing on $\cs' \times \cs$.
Then 
\[
 \| g \| := \sup \Big\{ |( g,h)|~:  \quad h \in \cs, \: \|h\|_{M^{-s}_{p',q'}} \le 1 \Big\}
\]
is an equivalent norm on $M^s_{p,q}$, see Feichtinger \cite{feichtinger} or Toft \cite{toftCont}.
In view of this equivalent norm our assumption on $T_f$ implies 
$\cl ({M}^{-s}_{p',q'}, {M}^{-s}_{p',q'})$.
Next we continue by complex interpolation. Let $0< \Theta < 1$.
It is known that 
\[
{M}^s_{p,q} = [{M}^{s_1}_{p_1,q_1}, {M}^{s_2}_{p_2,q_2}]_\Theta
\] 
if $1 \le p_1,q_1<\infty$, $1 \le p_2, q_2\le \infty$, $s_1,s_2 \in \re$ and  
\[
s:= (1-\Theta)s_1 + \Theta \, s_2\, , \quad \frac{1}{p} := \frac{1-\Theta}{p_1} + \frac{\Theta}{p_2}\, , \quad 
 \frac{1}{q} := \frac{1-\Theta}{q_1} + \frac{\Theta}{q_2}\, , 
\]
see Feichtinger \cite{feichtinger}.
Thanks to the interpolation property of the complex method 
we conclude
\[
 T_f \in \cl \Big( [\mathring{M}^{s}_{p,q}, {M}^{-s}_{p',q'}]_{1/2}, [\mathring{M}^{s}_{p,q}, {M}^{-s}_{p',q'}]_{1/2} \Big)\, .
\]
Because of $\mathring{M}^{s}_{p,q} = {M}^{s}_{p,q}$ if $\max(p,q)<\infty$ we find 
\[
T_f \in \cl \Big( {M}^{0}_{2,2}, {M}^{0}_{2,2}\Big) = \cl \Big( L_{2}, L_{2}\Big) \, .
\]
But this implies $f\in L_\infty$.
\end{proof}

\begin{proof} of Theorem \ref{mult2}. \\
{\em  Step 1.}
Sufficiency is covered by Lemma \ref{mult1}.
\\
{\em Step 2.} Necessity in case  $1\le p,q<\infty $ and $s\in \re$. 
In view of Lemma \ref{notw} the embedding 
$M^s_{p,q}\, \cdot \, M^s_{p,q}\hookrightarrow M^s_{p,q}$ implies 
$M^s_{p,q} \subset L_\infty$. 
\\
{\em Step 3.} To treat the remaining cases $\max (p,q) = \infty$ we argue by using explicit 
counterexamples.
\\
{\em Substep 3.1.} Let $1 \le p \le \infty$,  $s=0$ and $1 < q \le \infty$.
We assume that $M^0_{p ,q}$ is an algebra.
This implies the existence of a constant $c>0$ such that 
\be\label{ws-11}
\|\, f\, \cdot \, g \, \|_{M^0_{p,q}} \le c \, \|\, f\, \|_{M^0_{p,q}}\, \|\, g \, \|_{M^0_{p,q}}
\ee
holds for all $f,g \in M^0_{p,q}$.
Let
\[
 f (x) = \psi (x)\, \sum_{k=1}^\infty  a_k \, e^{ikx_1}\, , \qquad x=(x_1, \ldots \, , x_n) \in \R\, , 
\]
be as in \eqref{ws-10b}.
Then, as shown above, 
\[
 \|\, f\, \|_{M^0_{p,q}} = \| \psi \|_{L_p}\, \| (a_k)_k\|_{\ell_q}
\]
follows. Let 
\[
 f_N (x) := \psi (x)\, \sum_{k=1}^N a_k\,  e^{ikx_1}\, , \qquad x=(x_1, \ldots \, , x_n) \in \R\, , \quad  N \in \N\, .
\]
Obviously $f_N \in \cs$. We assume that  
\[
\supp \cf \psi \subset \{\xi: \max_{j=1, \ldots , n}\, |\xi_j|< \varepsilon\}\,  \qquad \mbox{with}\quad  \varepsilon < 1/4.
\]
Then, because of
\[
f_N (x)\, \cdot \, f_N (x) = \psi^2 (x) \, \sum_{m=2}^{2N} \Big(\sum_{k = 1}^{m-1} 
a_k  \, a_{m-k}  \Big)\, e^{imx}\, , 
\]
we conclude
\[
 \| f_N \, \cdot \, f_N \|_{M^0_{p,q}} = \|\psi^2\|_{L_p} \, \Bigg(
\sum_{m=2}^{2N} \Big|\sum_{k=1}^{m-1} a_k \, a_{m-k} \Big|^q\Bigg)^{1/q}\, .
\]
Inequality \eqref{ws-11} implies
\beqq
\Bigg(
\sum_{m=2}^{2N} \Big|\sum_{k=1}^{m-1} a_k \, a_{m-k} \Big|^q\Bigg)^{1/q}
\le c \, \frac{\|\psi\|^2_{L_p}}{\|\psi^2\|_{L_p}} \, \Big(\sum_{k=1}^{N} |a_k|^q\Big)^{2/q}\, .
\eeqq
Clearly, in case $q>1$ this is impossible in this generality.
Explicit counterexamples are given by
\[
a_k := k^{-1/q} \qquad \mbox{if}\quad 1 < q < \infty
\]
and 
\[
a_k =1 \qquad \mbox{if}\quad  q =\infty\, .
\]
In case $1<q< \infty$ \eqref{ws-11} yields
\[
 \| f_N \, \cdot \, f_N \|_{M^0_{p,q}} \asymp N^{1-1/q}\qquad \mbox{and}\qquad \| \, f_N \,  \|_{M^0_{p,q}}^2 \asymp (\log N)^{2/q}\, .
\]
For $q= \infty$ we obtain 
\[
 \| f_N \, \cdot \, f_N \|_{M^0_{p,\infty}} \asymp N \qquad \mbox{and}\qquad \| \, f_N \,  \|_{M^0_{p,\infty}}^2 \asymp 1\, .
\]
For $N \to \infty$ we find a contradiction in both situations.   
\\
{\em Substep 3.2.} Let $1 \le p \leq \infty$, $q=\infty$ and $0 < s \le n$.
We argue as in Substep 3.1 and assume $M^s_{p,\infty}$ is an algebra with respect to pointwise multiplication.
This leads to the existence of a constant $c>0$ such that 
\[ \|\, f\, \cdot \, g \, \|_{M^s_{p,\infty}} \le c \, \|\, f\, \|_{M^s_{p,\infty}}\, \|\, g \, \|_{M^s_{p,\infty}} \]
holds for all $f,g \in M^s_{p,\infty}$.
We choose 
\[ f(x)=g(x)=f_N(x) := \psi (x)\, \sum_{\|k\|_\infty \le N} a_k\,  e^{ikx}\,\, , \qquad x \in \R\, , \]
 and obtain
\beqq 
 \|\, f_N\, \|_{M^s_{p,q}} & = & \| \psi \|_{L_p}\, \Big( \sum_{\|k\|_\infty \le N} |a_k\, \langle k\rangle^s | ^q \Big)^{1/q} \, ,
\\
 \| f_N \, \cdot \, f_N \|_{M^s_{p,q}} & = & \|\psi^2\|_{L_p} \, \Bigg(
\sum_{\|m\|_\infty \le 2N} \langle m \rangle^{sq}\,  \Big|\sum_{{k: ~ \|k\|_\infty\le N \atop  \| m -k\|_\infty  \le N}} 
a_k \, a_{m-k} \Big|^q\Bigg)^{1/q} \, .
\eeqq
In case $s<n$ we choose $a_k :=1$ for all $k$ and obtain
\[
 \| f_N \, \cdot \, f_N \|_{M^s_{p,\infty}} \asymp N^{n+s} \qquad \mbox{and}\qquad \| \, f_N \,  \|_{M^s_{p,\infty}}^2 \asymp N^{2s}\, .
\]
This yields a contradiction if $s<n$.
For $s=n$ we consider $a_k := \langle k\rangle^{-n} $ for all $k$. This yields
\[
\log N\ls \, 
 \| f_N \, \cdot \, f_N \|_{M^n_{p,\infty}}  \qquad \mbox{and}\qquad \| \, f_N \,  \|_{M^n_{p,\infty}}^2 \asymp 1\, ,
\]
yielding a contradiction as well.
\\
{\em Substep 3.3.} Let $s<0$ and $1\le p,q\le \infty$.
We choose $a_k := \langle k\rangle^{2|s|} $ for all $k$ and obtain
\[
N^{3|s| + n+n/q}\ls \, 
 \| f_N \, \cdot \, f_N \|_{M^s_{p,q}}  \qquad \mbox{and}\qquad \| \, f_N \,  \|_{M^s_{p,q}}^2 \asymp N^{2|s| + 2n/q}\, .
\]
For $N \to \infty$ this implies $|s|+n \le n/q$. Since $|s|>0$ this is impossible. 
The proof is complete.
\end{proof}

\begin{cor}\label{mult4}
Let $1\le p,q\le \infty$ and $s \ge 0$. 
Then $M^s_{p,q}\cap M^0_{\infty,1}$ is an algebra with respect to pointwise multiplication and there 
exist a constant $c$ such that
\[
\| \, f \, \cdot \, g \, \|_{M^s_{p.q}} \le c\,  \Big(\| \, f \|_{M^0_{\infty,1}} \, \| \, g \|_{M^s_{p,q}}  +  \| \, f \|_{M^s_{p,q}} \, \| \, g \|_{M^0_{\infty,1}} \Big)
\]
holds for all $f,g \in M^s_{p,q}\cap M^0_{\infty,1}$.
\end{cor}

\begin{proof}
 The same arguments as in Lemma \ref{mult1} apply.
\end{proof}

\begin{rem}\label{mult5}
 \rm
Corollary \ref{mult4} has a  counterpart for Besov spaces.
Here one knows that $B^s_{p,q} \cap L_\infty$ is an algebra if $1 \le p,q\le \infty$ and $s>0$.
We refer to Peetre \cite[Thm.~11, page 147]{peetre} and to \cite[Thm.~4.6.4/2]{RS}.
\end{rem}


\subsection{More general products of functions} \label{general}


Here we consider the problem 

\[
 M^{s_1}_{p_1,q_1} \, \cdot \, M^{s_2}_{p_2,q} \hookrightarrow M^{s}_{p,q}\, .
\]
As a first result we mention a generalization of Lemma \ref{mult1}.

\begin{lem}\label{mult6}
Let $1\le p_1,p_2,q \le \infty$ and $s \ge 0$. 
We put $1/p:= (1/p_1) + (1/p_2)$. If $p \in [1,\infty]$, then 
there exists a constant $c$ such that
\[
\| \, f \, \cdot \, g \, \|_{M^s_{p.q}} \le c\,  \Big(\| \, f \|_{M^0_{p_1,1}} \, \| \, g \|_{M^s_{p_2,q}}  +  \| \, f \|_{M^s_{p_1,q}} \, \| \, g \|_{M^0_{p_2,1}} \Big)
\]
holds for all $f \in M^s_{p_1,q} \cap M^0_{p_1,1}$ and all $g \in M^s_{p_2,q} \cap M^0_{p_2,1}$.
\end{lem}

\begin{proof} 
We argue similar as above but using H\"older's inequality with respect to $p$ before applying the generalized Minkowski inequality. 
\end{proof}

\begin{rem}
 \rm
 Observe that $M^0_{p_1,1}, M^0_{p_2,1} \hookrightarrow M^0_{\infty,1} \hookrightarrow L_\infty$.
\end{rem}

\begin{lem}\label{mult7}
Let $1\le p_1,p_2,q \le \infty$ and $s \le  0$. 
We put $1/p:= (1/p_1) + (1/p_2)$. If $p \in [1,\infty]$, then 
there exists a constant $c$ such that
\be\label{ws-305}
 \| \, f \, \cdot \, g \, \|_{M^s_{p.q}}\le c\,  \|\, f\, \|_{M^{|s|}_{p_1,1}}  \, \|\, g\,  \|_{M^s_{p_2, q}}
\ee
holds for all $f \in M^{|s|}_{p_1,1}$, $g \in  M^s_{p_2,q}$ such that $g$ satisfies $g \in L_2^{\ell oc}$ and
\be\label{ws-306}
\int_{Q_k} |g(x)|^2 \, dx \le C \, (1+|k|)^{M} \, , 
\ee
for some  $C>0$ and $M>0$ independent of $k\in \Z$. 
\end{lem}

\begin{proof}
Point of departure is the formula \eqref{ws-14B}. Instead of the splitting in \eqref{ws-15B}
we use now the elementary inequality 
\[
1+|\eta|^2 \le 2 \,  (1+|\xi|^2)\, (1+|\xi - \eta|^2) 
\]
which implies
\[
(1+|\xi|^2)^{s/2} \le 2^{|s|/2} \,  (1+|\xi - \eta|^2)^{|s|/2}\, (1+|\eta|^2)^{s/2}\, . 
\]
This leads to the estimate
\beqq
&& \hspace*{-0.7cm}  \| \, f \, \cdot \, g \, \|_{M^s_{p.q}} 
\\
& \ls &  
\Big\{ \int \Big[ \int \Big|  \int V_{\varphi} f(x,\xi-\eta)\, \langle \xi - \eta \rangle^{|s|} \,  V_\varphi g (x,\eta) \, 
\langle \eta \rangle^{s} \,  d\eta\, \Big|^p dx\Big]^{q/p} \, d\xi \Big\}^{1/q}
\\
& = & 
\Big\{ \int \Big[ \int \Big|  \int V_{\varphi} f(x,\tau)\, \langle \tau \rangle^{|s|} \,  V_\varphi g (x,\xi -\tau) \, 
\langle \xi - \tau \rangle^s \,  d\tau\, \Big|^p dx\Big]^{q/p} \, d\xi \Big\}^{1/q}\, .
\eeqq
We continue by 
applying the generalized Minkowski inequality and  H\"older's inequality (with respect to $p$) and obtain
\beqq
&& \hspace*{-0.7cm}  \| \, f \, \cdot \, g \, \|_{M^s_{p,q}} 
\\
& \ls &  \int \Big\{ \int \Big[  \| V_{\varphi} f(x,\tau)\, \langle \tau \rangle^{|s|} \|_{L_{p_1}} \, 
\, \|V_\varphi g (x,\xi - \tau)\, \langle \xi - \tau \rangle^s \|_{L_{p_2}} \Big]^{q} \, d\xi \Big\}^{1/q} d\tau
\\
& \ls &  \int  \| V_{\varphi} f(x,\tau)\, \langle \tau \rangle^{|s|} \|_{L_{p_1}} d\tau \,\,  \| g\|_{M_{p_2,q}^s}
\\
& \ls &
\|\, f\, \|_{M^{|s|}_{p_1,1}}  \, \|\, g\, \|_{M^s_{p_2,q}}
\, .
\eeqq
\end{proof}

\begin{rem}
 \rm
Observe that $M^{|s|}_{p_1,1} \hookrightarrow M^{|s|}_{\infty,1} \hookrightarrow L_\infty$. 
In addition we would like to mention that the constant $c$ in \eqref{ws-305} does not depend on the 
constant $C$ in \eqref{ws-306}.
\end{rem}

We recall a final result of Cordero  and Nicola \cite{cordero} concentrating on $s=0$.
These  authors study  $M^{0}_{p_1,q_1} \, \cdot \, M^{0}_{p_2,q_2} \hookrightarrow M^{0}_{p,q}$.

\begin{prop}\label{coni}
 Let $1\le p_1,p_2,q_1,q_2 \le \infty$. Then
$M^0_{p_1,q_1} \, \cdot \, M^0_{p_2,q_2} \hookrightarrow M^0_{p,q}$ holds if and only if 
\[
\frac 1p \le \frac{1}{p_1} +  \frac{1}{p_2} \qquad \mbox{and}\qquad 1 + \frac 1q \le \frac{1}{q_1} +  \frac{1}{q_2}\, .
\]
\end{prop}

\begin{rem}
 \rm 
{(i)}  Proposition \ref{coni} shows that in case $s=0$ in Lemma \ref{mult7} we proved an optimal estimate.\\ 
{(ii)} Necessity of the restrictions in Proposition \ref{coni} is shown by studying products of Gaussian functions.
For extensions of Proposition \ref{coni} to the case of products with more than two factors we refer to Guo et all \cite{guo} and 
Toft \cite{toftCont}.
\end{rem}


\subsection{Products of a distribution with a  function} \label{general2}


Up to now  we considered only products of either $L_\infty$-functions or $L_2 ^{\ell oc}$-functions with $L_\infty$-functions. 
But now we turn to the product of a distribution with a function which is not assumed to be $C^\infty$.
This requires a definition.

\subsection*{The definition of the product in $\cs'$}

Let $\psi \in \cs$ be a  function in $C_0^\infty $ such that $\psi (\xi)=1$ in a neighbourhood of the origin.
We define 
\[ S^j  f (x) = \cfi [ \psi (2^{-j} \xi ) \, \cf  f(\xi ) ]
( x ), \quad j = 0, 1, \ldots  \, . \]
The Paley-Wiener theorem tells us that $S^j f$ is an
entire analytic function of exponential type. 
Hence, if $f, \: g \in \cs'$ the products $S^j f \cdot
S^j g$ makes sense for any $j$.
Further, 
\[ \lim_{j \rightarrow \infty} \cfi [ \psi (2^{-j} \xi ) \, \cf  f(\xi ) ]
( \cdot ) = f \qquad (\mbox{convergence in } \cs') \]
for any $f \in \cs'$.

\begin{defn}
Let $f, \: g \in {\cs}'.$ We define
\[ f\, \cdot\,  g = \lim_{j \rightarrow \infty} \, S^j f \cdot S^j g \]
whenever the limit on the right-hand side exists in $\cs'$. We call $f
\cdot g$ the product of $f$ and $g$. 
\end{defn}

\begin{rem}
 \rm
 In defining the product we followed a usual practice, see, e.g., \cite{peetre}, \cite[2.8]{triebel}, \cite{Jo1,Jo2} and \cite[4.2]{RS}.
 For basic properties of this notion we refer to 
 \cite{Jo1,Jo2} and \cite[4.2]{RS}.
\end{rem}

\begin{thm}\label{mult8}
Let $1\le p_1,p_2,q \le \infty$ and $s \le  0$. 
We put $1/p:= (1/p_1) + (1/p_2)$. 
If $p \in [1,\infty]$, then 
there exists a constant $c$ such that
\[
 \| \, f \, \cdot \, g \, \|_{M^s_{p,q}}\le c\,   \|\, f\, \|_{M^{|s|}_{p_1,1}}  \, \|\, g\,  \|_{M^s_{p_2, q}}) 
\]
holds for all $f \in M^{|s|}_{p_1,1}$ and  $g \in  M^s_{p_2,q}$.
\end{thm}

\begin{proof}
We have to show that the limit of $(S^j f \, \cdot \, S^jg)_j$ exists in $\cs'$.
The remaining assertions, $ \lim_{j\rightarrow \infty} S^j f \, \cdot \, S^jg \in M^s_{p,q} $ and the norm estimates will follow by employing the Fatou property, 
see Lemma \ref{basic},  and Lemma \ref{mult7}.
\\
{\em Step 1.} Let $1\le q<\infty$.
We have 
 \[
  \lim_{j\to \infty} \| S^j f- f\|_{M^s_{p,q}}= 0 \qquad \mbox{for all}\quad f \in M^s_{p,q}\, .
 \]
In addition it is easily seen that
\be\label{ws-30}
 \sup_{j\in \N_0} \| S^j f\|_{M^s_{p,q}}\le \|\cfi \psi \|_{L_1} \, \|\, f \, \|_{M^s_{p,q}}  
\ee
holds for all $ f \in M^s_{p,q}$.
Hence,  we conclude by means of Lemma \ref{mult7}
\beqq
&& \hspace*{-1.2cm}  \|S^k f \, S^k g - S^j f \, S^j g  \|_{M^s_{p,q}}  \\
& \le & \|( S^k f  - S^j f) \, S^k g  \|_{M^s_{p,q}}  + \|S^j f \, (S^k g - S^j g)  \|_{M^s_{p,q}}
\\
& \le & c\, 
 ( \|\, S^k g\, \|_{M^s_{p_2,q}}  \, \|\, S^k f  - S^j f\, \|_{M^{|s|}_{p_1,1}} 
+    \|\, S^k g - S^j g\, \|_{M^s_{p_2,q}}  \, \|\,  S^j f\, \|_{M^{|s|}_{p_1,1}}) 
\eeqq
the convergence of $(S^kf \, \cdot S^k g)_k$ in $M^s_{p,q}$ and therefore in $\cs'$, see Lemma \ref{basic}.
\\
{\em Step 2.} Let $q= \infty$ and suppose $p=1$.
Let $\psi, \psi^* \in C_0^\infty$ be  functions such that $\psi (\xi) = 1$, $|\xi|\le 1 $, $\psi (\xi) =0$ if $|\xi|>3/2$ and 
 $\psi^* (\xi) = 1$, $|\xi|\le 6 $. Then checking the Fourier support of the product $ S^k f \, S^k g$ and using linearity of $\cf$  we conclude
 \beqq
  \Big\langle S^k f \, S^k g  &- & S^j f \, S^j g , \varphi \Big\rangle  \\
  & = & \Big\langle S^k f \, S^k g  -  S^j f \, S^j g , \cfi [(\psi^* (2^k\xi) - \psi^* (2^j \xi))\cf \varphi(\xi)](\, \cdot \, ) \Big\rangle  \, .
 \eeqq
For brevity we put 
\[
h_1:=  S^k f \, S^k g  -  S^j f \, S^j g \qquad \mbox{and}\qquad h_2:= \cfi [(\psi^* (2^k\xi) - \psi^* (2^j \xi))\cf \varphi(\xi)](\, \cdot \, )\, .
\]
$h_1,h_2$ are smooth functions with compactly supported Fourier transform.
Hence
\[
 h_1 = \sum_{k \in I_1} \Box_k h_1 \qquad \mbox{and}\qquad h_2 = \sum_{k \in I_2} \Box_k h_2\, ,   
\]
where $I_1,I_2$ are finite subsets of $\Z$.
This allows us to rewrite $\Big\langle S^k f \, S^k g  -  S^j f \, S^j g , \varphi \Big\rangle$ as follows
\beqq
 \Big\langle S^k f \, S^k g  -  S^j f \, S^j g , \varphi \Big\rangle & = &  
 \sum_{k \in I_1}\sum_{\ell \in I_2} \int  \Box_k h_1(x) \, \Box_\ell h_2 (x)\, dx
 \\
& = &  
 \sum_{\ell  \in I_2}\sum_{k \in I_1: \: Q_k\cap Q_\ell \neq \emptyset} \int  \Box_k h_1(x) \, \Box_\ell h_2 (x)\, dx 
 \, .
\eeqq
Application of H\"older's inequality yields
\beq\label{ws-34}
\Big| \Big\langle S^k f \, S^k g  -  S^j f \, S^j g , \varphi \Big\rangle\Big| & \le &   2^n \sup_{k \in \Z} \langle k\rangle^{s} \|\,   \Box_k h_1\, \|_{L_{p_1}} \, 
\Big( \sum_{\ell \in \Z} \langle \ell \rangle^{-s}\| \, \Box_\ell h_2 \, \|_{L_{p_2}}\Big)
\nonumber
\\
& \le &   2^n \,  \|\,  h_1\, \|_{M^s_{p_1,\infty}} \, \| \,  h_2 \, \|_{M^{-s}_{p_2,1}}
 \, .
\eeq
By means of Lemma \ref{mult7} and \eqref{ws-30} we know that 
\beqq
\|\,  h_1\, \|_{M^s_{p_1,\infty}} & = &  \|\, S^k f \, S^k g  -  S^j f \, S^j g \, \|_{M^s_{p_1,\infty}} 
\\
& \le &  
c_1\, \sup_{j\in \N_0}   \|\, S^j g\, \|_{M^s_{p_2,q}}  \, \|\, S^jf\, \|_{M^{|s|}_{p_1,1}}  
\\
& \le &  
c_2\,    \|\, g\, \|_{M^s_{p_2,q}}  \, \|\, f\, \|_{M^{|s|}_{p_1,1}} \, .
\eeqq
On the other hand, if $j \le k$,  a standard Fourier multiplier argument yields
\beqq
\| \,  h_2 \, \|_{M^{-s}_{p_2,1}} & = & \| \,  \cfi [(\psi^* (2^k\xi) - \psi^* (2^j \xi))\cf \varphi(\xi)](\, \cdot \, ) \, \|_{M^{-s}_{p_2,1}} 
\\
&\le & C \, \sum_{A\, 2^j \le |\ell| \le B \, 2^k} \langle \ell \rangle^{-s}\, \| \,  \Box_\ell \varphi \, \|_{L_{p_2}} 
\eeqq
for appropriate positive constants $A,B,C$ independent of $j,k$ and $\varphi$.
Since $\varphi \in \cs \subset M^{-s}_{p_2,1}$ we conclude that the right-hand side tends to $0$ if $j \to \infty$.
This finally proves
\[
 \Big|\Big\langle S^k f \, S^k g  -  S^j f \, S^j g , \varphi \Big\rangle\Big| < \varepsilon \qquad \mbox{if}\quad j,k \ge j_0 (\varepsilon)\, . 
\]
Hence $(S^k f \, S^k g)_k$ is weakly convergent in $\cs'$.
Now, Lemma \ref{mult7} yields the claim also for $q=\infty$. 
\\
{\em Step 3.} Let $q= \infty$ and suppose $1 < p \le \infty$. We employ \eqref{ws-34} with $p_1 = \infty$ and $p_2 =1$ and afterwards Proposition \ref{einbettung1}.
It follows
\beqq
\Big| \Big\langle S^k f \, S^k g  -  S^j f \, S^j g , \varphi \Big\rangle\Big|    
& \le &   2^n \,  \|\,  h_1\, \|_{M^s_{\infty,\infty}} \, \| \,  h_2 \, \|_{M^{-s}_{1,1}}
\\
& \le &   c_1 \,  \|\,  h_1\, \|_{M^s_{p,q}} \, \| \,  h_2 \, \|_{M^{-s}_{1,1}}
 \, .
\eeqq
Now we can argue as in Step 2.
\end{proof}

\begin{rem}
 \rm
For a partial result concerning Theorem \ref{mult8} we refer to \\
Feichtinger \cite{feichtinger}.
\end{rem}

\subsection{One example}

We consider the Dirac $\delta$  distribution.
Since 
\[
\cf \delta (\xi) = (2\pi)^{-n/2}\, , \qquad \xi \in \R\, , 
\]
it is easily seen that $\delta \in M^0_{p,\infty}$ for all $p$.
Also not difficult to see is that $M^0_{1,\infty}$ is the smallest space of type $M^s_{p,q}$ to which $\delta $ belongs to.
Thm. \ref{mult8} yields
\[
 \| \, f \, \cdot \, \delta \, \|_{M^0_{p,\infty}}\le c\,   \|\, \delta\, \|_{M^0_{p,\infty}}  \, \|\, f\, \|_{M^0_{\infty,1}} 
\]
with some $c$ independent of $f \in M^0_{\infty,1}$.
With other words, we can multiply $\delta$ with a modulation space $M^s_{p,q}$ if this space is embedded into $C_{ub}$, see Cor. \ref{einbettung2}.
This looks reasonable.


\subsection{The second method}


Finally we would like to investigate also the cases $\min(s_1 , s_2) \le  n/q'$.
For dealing with this special situation we turn to a different  method which will allow a better localization in the Fourier image.
Therefore we shall work with the frequency-uniform decomposition $(\sigma_k)_k$.
Recall that $\supp \sigma_k \subset Q_k := \{\xi \in\R : -1\leq \xi_i-k_i\leq 1, \, i=1,\ldots,n\}$. 
For brevity we put 
\[ f_k (x):= \cfi [\sigma_k (\xi) \cf f (\xi)](x) \, , \qquad x \in \R\, , \quad k \in \Z \, . \]
Then, at least formally,  we have the following representation of the product $f\cdot g$ as
\[ f\cdot g = \sum_{k,l\in \Z} f_k \cdot g_l. \]
In what follows we shall study bounds for related partial sums.

\begin{lem} \label{algebra10}
Let $1\leq p_1,p_2\leq \infty$, $1 < q \le \infty$ and $s_0 , s_1,s_2 \in \re$.  
Define $p$ by $\frac{1}{p}:= \frac{1}{p_1}+\frac{1}{p_2}$. 
If $p \in [1,\infty]$, $0 \le s_0 \le \min(s_1, s_2)$ and  $s_2 + s_1 - s_0 > n/q'$, 
then there exists a constant $c$ such that 
\[ 
\|\sum_{k,l\in \Z} f_k \cdot g_l\|_{{M}_{p,q}^{s_0}} \leq c \, \| f \|_{{M}_{p_1,q}^{s_1}}\,  \| g\|_{{M}_{p_2,q}^{s_2}} 
\]
holds for all $f, g\in \cs'$ such that $\supp \cf f$ and $\supp \cf g$ are compact.
The constant $c$ is independent from $\supp \cf f$ and $\supp \cf g$, respectively.
\end{lem}
	
\begin{proof}
Later on we shall use the same strategy of proof as below in slightly different situations.
For this reason and later use we shall take care of all constants showing up in our estimates below.
\\
{\em Step 1.} Preparations.
Determining the Fourier support of $f_j \cdot g_l$ we see that
\begin{eqnarray*}  \supp \cf (f_j\cdot  g_l)
& = & \supp (\cf f_j \ast \cf g_l) \\  & \subset& \{\xi\in \R :  j_i+l_i-2 \leq \xi_i \leq j_i+l_i+2, \, i=1,\ldots,n\}. 
\end{eqnarray*}
Hence, the term $\cfi  (\sigma_k \cf (f_j \cdot g_l))$  vanishes if $\|k - (j+l)\|_\infty \ge  3$. 
In addition, since  $\supp \cf f$ and $\supp \cf g$ are compact,  the sum $\sum_{j,l \in \Z} f_j \cdot g_l$ is a finite sum.
We obtain

\begin{eqnarray*}
\sigma_k \cf (f\cdot g) & = & \sigma_k \cf \Big(\sum_{j,l\in\Z} f_j\cdot g_l \Big) 
 =   \sigma_k \cf \Big(\sum_{\substack{j,l \in \Z, \\ k_i-3<j_i+l_i < k_i +3, \\ 
i=1,\ldots, n}} f_j\cdot  g_l \Big)  \\
& \stackrel{[r=j+l]}{=} & \sum_{\substack{r \in \Z, \\ k_i-3<r_i < k_i +3, 
\\ i=1,\ldots, n}}  \sum_{l\in\Z}\, 
\sigma_k \cf \big( f_{r-l}\cdot  g_l \big) \, .
\end{eqnarray*}

Consequently 
\begin{eqnarray*}
\left\|\cfi \big(\sigma_k \cf (f\cdot g)\big) \right\|_{L_p} 
& \leq & \sum_{\substack{r \in \Z, \\ k_i-3<r_i < k_i +3, \\ i=1,\ldots, n}} \sum_{l\in\Z} \| \cfi \big( \sigma_k \cf ( f_{r-l}\cdot  g_l) 
\big)\|_{L_p} \\
& \stackrel{[t=r-k]}{=} & \sum_{\substack{t \in \Z, \\ -3<t_i < 3, \\ i=1,\ldots, n}} \sum_{l\in\Z} \| \cfi 
\big( \sigma_k \cf ( f_{t-(l-k)}\cdot  g_l) \big)\|_{L_p}.
\end{eqnarray*}
{\em Step 2.} Norm estimates.
These preparations yield the following  estimates
\begin{eqnarray*}
 \Big( \sum_{k\in\Z} && \hspace{-0.9cm}  \langle k\rangle^{s_0 q} \|\cfi \big(\sigma_k \cf (f\cdot g) \big)\|_{L_p}^q \Big)^{\frac{1}{q}} 
\\
& \leq & \Bigg( \sum_{k\in\Z}  \langle k\rangle^{s_0q} \Bigg[ \sum_{\substack{t \in \Z, \\ -3<t_i < 3, \\ i=1,\ldots, n}}
 \sum_{l\in\Z} \| \cfi \big(\sigma_k \cf ( f_{t-(l-k)} \cdot g_l) \big)\|_{L_p} \Bigg]^q \Bigg)^{\frac{1}{q}} \\
& \leq & \sum_{\substack{t \in \Z, \\ -3<t_i < 3, \\ i=1,\ldots, n}} \left( \sum_{k\in\Z}  \langle k\rangle^{s_0q}
\left[ \sum_{l\in\Z} \| \cfi \big(\sigma_k \cf ( f_{t-(l-k)} \cdot g_l) \big)\|_{L_p} \right]^q \right)^{\frac{1}{q}} \, .
\end{eqnarray*}
Observe
\begin{eqnarray*}
 \| \cfi \big(\sigma_k \cf ( f_{t-(l-k)} \cdot g_l) \big)\|_{L^p} & = & (2\pi)^{-n/2} \| (\cfi \sigma_k) *  (f_{t-(l-k)} \cdot g_l) \, \|_{L_p}
 \\
 &\le & (2\pi)^{-n/2} \| \cfi \sigma_k \, \|_{L^1} \, \|\,  f_{t-(l-k)} \cdot g_l \, \|_{L_p}
\\
& =  & (2\pi)^{-n/2} \| \cfi \sigma_0 \, \|_{L^1} \, \|\,  f_{t-(l-k)} \cdot g_l \, \|_{L_p} \, ,
 \end{eqnarray*}
where we used Young's inequality. We put $c_1:= (2\pi)^{-n/2} \| \cfi \sigma_0 \, \|_{L_1}$. This implies
\begin{eqnarray*}
\Big( \sum_{k\in\Z}  \langle k\rangle^{s_0 q} && \hspace{-0.7cm} \|\cfi \big(\sigma_k \cf (f\cdot g) \big)\|_{L_p}^q \Big)^{\frac{1}{q}} 
\\
& \leq & c_1\,  \sum_{\substack{t \in \Z, \\ 
-3<t_i < 3, \\ i=1,\ldots, n}} \left( \sum_{k\in\Z}  \langle k\rangle^{s_0q} \left[ \sum_{l\in\Z} 
\|f_{t-(l-k)}\cdot  g_l\|_{L_p} \right]^q \right)^{\frac{1}{q}}\, .
\end{eqnarray*}
We continue by using H\"older's inequality to get
\begin{eqnarray*} 
 \Big( \sum_{k\in\Z}   \langle k\rangle^{s_0q} && \hspace{-0.7cm} \|\cfi \big(\sigma_k \cf (f\cdot g) \big)\|_{L_p}^q \Big)^{\frac{1}{q}} 
 \\
& \leq & c_2 \, \max_{\substack{t \in \Z, \\ -3<t_i < 3, \\ i=1,\ldots, n}} \left( \sum_{k\in\Z}  \langle k\rangle^{s_0q} 
\left[ \sum_{l\in\Z} \|f_{t-(l-k)}\|_{L_{p_1}} \|g_l\|_{L_{p_2}} \right]^q \right)^{\frac{1}{q}} \hspace{1cm} 
\end{eqnarray*}
with $c_2:= c_1 \, 5^n$. Since $s_0\ge 0$ elementary calculations  yield

\begin{eqnarray*}
  \langle k\rangle^{s_0}
\Big[ \sum_{l\in\Z} && \hspace{-0.7cm} \|f_{t-(l-k)}\|_{L_{p_1}}\,  \|g_l\|_{L_{p_2}} \Big] 
\\
& \leq & 2^{s_0}\,  \sum_{\substack{l\in\Z, \\
 |l|\leq |l-k|}}  \langle k-l\rangle^{s_0} \|f_{t-(l-k)}\|_{L_{p_1}}   \|g_l\|_{L_{p_2}}  
\\
& & \qquad + 2^{s_0} \sum_{\substack{l\in\Z, \\ |l-k|\leq |l|}}   \|f_{t-(l-k)}\|_{L_{p_1}}  \langle l\rangle^{s_0} 
\|g_l\|_{L_{p_2}}  \,.
\end{eqnarray*}
Both parts of this right-hand side will be  estimated separately. We put 
\begin{eqnarray*}
 S_{1,t,k} & := &  \sum_{\substack{l\in\Z, \\
 |l|\leq |l-k|}}  \langle k-l\rangle^{s_1} \|f_{t-(l-k)}\|_{L_{p_1}}  \langle l\rangle^{s_2} \|g_l\|_{L_{p_2}} \, 
\langle k-l\rangle^{s_0 - s_1} \, \langle l\rangle^{-s_2} \, ; 
\\
S_{2,t,k} & := & \sum_{\substack{l\in\Z, \\ |l-k|\leq |l|}}  \langle k-l\rangle^{s_1} \, 
\|f_{t-(l-k)}\|_{L_{p_1}}  \langle l\rangle^{s_2} \, \|g_l\|_{L_{p_2}} \, \langle k-l\rangle^{- s_1} \, 
\langle l\rangle^{s_0-s_2} \, .
\end{eqnarray*}
With $\frac{1}{q}+\frac{1}{q'}=1$  we find
\begin{eqnarray*}
S_{1,t,k} & \stackrel{[j=l-k]}{=} &  \sum_{\substack{j\in\Z, \\ |j+k|\leq |j|}}   \langle j\rangle^{s_0} \|f_{t-j}\|_{L_{p_1}} 
 \|g_{j+k}\|_{L_{p_2}}  
\\
& \leq &   \Big( \sum_{\substack{j\in\Z, \\ |j+k|\leq |j|}} ( \langle j\rangle^{s_1} \|f_{t-j}\|_{L_{p_1}}  
\langle  j+k\rangle^{s_2} \|g_{j+k}\|_{L_{p_2}}   )^{q}\Big)^{1/q}  \\
			& & \qquad \qquad \qquad \times \quad  \Big( \sum_{\substack{j\in\Z, \\ 
|j+k|\leq |j|}} (\langle j \rangle^{s_0 - s_1} \, \langle j+k \rangle^{-s_2})^{q'}  \Big)^{\frac{1}{q'}} \, .
\end{eqnarray*}
{\em Substep 2.1.} 
Our assumptions $s_0 \le s_1$, $s_2 \ge 0$ and $s_1 + s_2 -s_0 >n/q'$ imply
\begin{equation*} 
\Big( \sum_{\substack{j\in\Z, \\ 
|j+k|\leq |j|}} \Big|  \langle j \rangle^{s_0 - s_1} \, \langle j+k \rangle^{-s_2} \Big|^{q'} \Big)^{\frac{1}{q'}} \le 
\Big( \sum_{m \in\Z}  \langle m \rangle^{(s_0 -s_1 -s_2) q'} \Big)^{\frac{1}{q'}} =: c_3 <\infty \, .
\end{equation*}
Inserting this in our previous estimate we obtain
\begin{eqnarray*}
\Big(\sum_{k \in \Z} S_{1,t,k}^q\Big)^{1/q} & \le  & c_3\,  \Bigg( \sum_{k\in\Z} 
\sum_{\substack{j\in\Z, \\ |j+k|\leq |j|}} \langle j \rangle^{s_1q}  \|f_{t-j}\|_{L^{p_1}}^q 
\langle j+k \rangle^{s_2q} \|g_{j+k}\|_{L^{p_2}}^q \Bigg)^{1/q}
\\
& \leq &  c_3 \, \Bigg( \sum_{j\in\Z} 
\langle j \rangle^{s_1q} \|f_{t-j}\|_{L^{p_1}}^q \sum_{k\in\Z} \langle j+k \rangle^{s_2q} 
\|g_{j+k}\|_{L^{p_2}}^q \Bigg)^{\frac{1}{q}} 
\, .
\end{eqnarray*}
Because of $1+|j|^{2} \le  1+ 8n + |j-t|^{2} $ we know
\begin{equation*} 
\max_{\substack{t \in \Z, \\ -3<t_i < 3, \\ i=1,\ldots, n}} \sup_{j\in\Z} 
\frac{\langle j \rangle^{s_1}}{\langle j-t \rangle^{s_1}}  \leq  (1+ 8n)^{s_1/2}  =: c_4 < \infty \, .
\end{equation*}
This implies 
\be\label{wsb}
\Big(\sum_{k \in \Z} S_{1,t,k}^q\Big)^{1/q} \le c_3 \, c_4 \, \|g\|_{{M}^{s_2}_{p_2,q}} \|f\|_{{M}^{s_1}_{p_1,q}},
\ee
where $c_3$, $c_4$ are independent of $f,g$ and $t$.
\\
{\em Substep 2.2.} Because of  $0 \le s_0 \le s_1$,   $s_0 \le s_2$ and $s_1 + s_2 -s_0 >n/q'$ we conclude 
\begin{equation*} 
\Big( \sum_{\substack{l\in\Z, \\ 
|l-k|\leq |l|}} \Big|  \langle k-l \rangle^{- s_1} \, \langle l \rangle^{s_0-s_2} \Big|^{q'} \Big)^{\frac{1}{q'}} \le 
\Big( \sum_{m \in\Z}  \langle m \rangle^{(s_0 -s_1 -s_2) q'} \Big)^{\frac{1}{q'}} =: c_5 <\infty \, .
\end{equation*}
This leads to  the estimate
\be\label{ws}
\Big(\sum_{k \in \Z} S_{2,t,k}^q\Big)^{1/q} \le c_5 c_6\,  \|g\|_{{M}^{s_2}_{p_2,q}} \|f\|_{{M}^{s_1}_{p_1,q}} 
\ee
with some constants $c_6$ independent from $f$ and $g$.
Combining the inequalities \eqref{wsb} and \eqref{ws}  we have proved the claim.
\end{proof}

\begin{rem}
 \rm
 Some basic ideas of the above proof are taken over from  \cite{brs}, see also \cite{rrs}.
\end{rem}

Of course the above method of proof works as well for $q=1$.
But all spaces $M^s_{p,1}$, $s\ge 0$, are algebras.

\begin{thm} \label{algebra12}
Let $1\leq p, p_1,p_2\leq \infty$ and $s_0, s_1, s_2 \in \re$.  
 Let $1/p \le (1/p_1) + (1/p_2)$,  $1 < q \le \infty$,   $0 \le s_0  \le \min(s_1, s_2)$ and $s_1 + s_2-s_0>n/q'$. 
There exists a constant $c$ such that 
\[ 
\|\, f \cdot g \, \|_{{M}_{p,q}^{s_0}} \leq c \, \| f \|_{{M}_{p_1,q}^{s_1}}\,  \| g\|_{{M}_{p_2,q}^{s_2}} 
\]
holds for all $f\in {M}_{p_1,q}^{s_1}$ and all $g \in {M}_{p_2,q}^{s_2}$.
\end{thm}
	
\begin{proof}
We only comment on the case $1/p = (1/p_1) + (1/p_2)$, see Corollary \ref{einbettung}. 
It will be enough to prove the weak  convergence of $(S^kf \cdot S^k g)_k$ in $\cs'$. The claimed estimate will then follow 
from Lemma \ref{algebra10}. We employ the method and the notation  used in proof of Thm. \ref{mult8} (Steps 2 and 3). There we have proved
\[
\Big| \Big\langle S^k f \, S^k g  -  S^j f \, S^j g , \varphi \Big\rangle\Big|  
\le    c_1 \,  \|\,  h_1\, \|_{M^{s_0}_{p,q}} \, \| \,  h_2 \, \|_{M^{-s_0}_{1,1}}
\]
with $c_1$ independent of $f,g,k$ and $j$.
By means of Lemma \ref{algebra10} we know the uniform boundedness of 
$\|\,  h_1\, \|_{M^{s_0}_{p,q}}$ in $k$ and $j$.
The estimate of $\| \,  h_2 \, \|_{M^{-s_0}_{1,1}}$ can be done as above. It follows
\[
 \| \,  h_2 \, \|_{M^{-s_0}_{1,1}} \le \varepsilon
\]
if $j,k\ge j_0(\varepsilon)$.
This guarantees the  weak convergence of $(S^kf \cdot S^k g)_k$ in $\cs'$. 
\end{proof}

Our sufficient conditions are not far away from necessary conditions.

\begin{lem}
Let $1\leq p_1,p_2, p, q\leq \infty$ and $s_0, s_1, s_2 \in \re$.  
Suppose that there exists a constant $c$ such that
\be\label{ws-311}
\|\, f \cdot g \, \|_{{M}_{p,q}^{s_0}} \leq c \, \| f \|_{{M}_{p_1,q}^{s_1}}\,  \| g\|_{{M}_{p_2,q}^{s_2}} 
\ee
holds for all $f, g \in \cs$.
\\
{\rm (i)}  It follows  $s_0 \le \min (s_1,s_2)$, $s_1 + s_ 2 \ge 0$ and $s_1 + s_2 -s_0 \ge n/q'$.
\\
{\rm (ii)} If $1 \le p_2 =p < \infty$ and $1\le q < \infty$, then 
either $q=1$ and $s_1 \ge 0 $ or $1 < q < \infty $ and $s_1 >n/q'$.
\end{lem}

\begin{proof}
Part (ii) is an immediate consequence of Lemma \ref{notw}.
Concerning the proof of (i) we shall work with the same test functions as  used in Step 2 of the proof of 
Corollary \ref{einbettung2}, see \eqref{ws-10b}.
\\
{\em Step 1.} We choose $a_k := \delta_{k,\ell}$, $k \in \Z$,  for a fixed given $\ell \in \Z$ and put $b_k:= \delta_{k,0}$, $k \in \Z$. 
Then we define
\[
 f(x) := \psi (x) \, e^{i \ell x} \qquad \mbox{and} \qquad g(x):= \psi (x)\, . 
\]
We obtain
\[
\|f\|_{M^{s_1}_{p_1,q}} \, \cdot \, \|g\|_{M^{s_2}_{p_2,q}} = \|\psi\|_{L_{p_1}}\, \|\psi\|_{L_{p_2}} \, \langle \ell\rangle^{s_1} 
\]
as well as 
\[
 \|f\, \cdot \, g \|_{M^{s_0}_{p,q}} = \|\psi^2\|_{L_{p}} \, \langle \ell\rangle^{s_0} \, .
\]
Hence, \eqref{ws-311} implies $s_0 \le s_1$. Interchanging the roles of $f$ and $q$ leads to the conclusion $s_0 \le s_2$.
\\
{\em Step 2.} Let $\ell \in \Z$ be fixed.
 We choose $a_k := \delta_{k,\ell}$, $k \in \Z$, and  $b_k:= \delta_{k,-\ell}$, $k \in \Z$. 
Then we define
\[
 f(x) := \psi (x) \, e^{i \ell x} \qquad \mbox{and} \qquad g(x):= \psi (x)\,  e^{-i \ell x} \, . 
\]
It follows
\[
\|f\|_{M^{s_1}_{p_1,q}} \, \cdot \, \|g\|_{M^{s_2}_{p_2,q}} = \|\psi\|_{L_{p_1}}\, \|\psi\|_{L_{p_2}} \, \langle \ell\rangle^{s_1+s_2} 
\]
as well as 
\[
 \|f\, \cdot \, g \|_{M^{s_0}_{p,q}} = \|\psi^2\|_{L_{p}}  \, .
\]
Hence, \eqref{ws-311} implies $s_1+ s_2\ge 0$. 
\\
{\em Step 3.} Let $\varepsilon_1, \varepsilon_2 \ge 0$.
These two numbers will be chosen such that
\[
\min(s_1 + \varepsilon_1 +  n/q, s_2  + \varepsilon_2+ n/q) >0
\qquad \mbox{and}\qquad s_0 + \varepsilon_2 + \varepsilon_1 + n>0\, .
\]
We choose $a_k := \langle k \rangle^{\varepsilon_1}$, $k \in \Z$,   and $b_k:= \langle k \rangle^{\varepsilon_2}$, $k \in \Z$. 
Then we define
\[
 f(x) := \psi (x) \, \sum_{\|k\|_\infty \le N} a_k \, e^{i k x} \qquad \mbox{and} \qquad g(x):= \psi (x) \, \sum_{\|k\|_\infty \le N} b_k \, e^{i k x} \, . 
\]
By means of the same arguments as used in Substep 3.1 of the proof of Theorem \ref{mult2}, we conclude
\[
 \|f\|_{M^{s_1}_{p_1,q}}  \asymp  N^{s_1 + \varepsilon_1 + n/q} \qquad \mbox{and}\qquad 
  \, \|g\|_{M^{s_2}_{p_2,q}}  \asymp  N^{s_2 + \varepsilon_2 + n/q} \, .
\]
In addition we have 
\beqq
\| f \, \cdot \, g \|_{M^{s_0}_{p,q}} &\asymp  &
\Bigg(
\sum_{\|m\|_\infty \le 2N} \langle m \rangle^{s_0q}\,  \Big|\sum_{{k: ~ \|k\|_\infty\le N \atop  \| m -k\|_\infty  \le N}} 
a_k \, b_{m-k} \Big|^q\Bigg)^{1/q}
\\
&\ge  & \frac{1}{2n}\, 
\Bigg(\sum_{\|m\|_\infty \le N} \langle m \rangle^{(s_0 + \varepsilon_2)q}\,  \Big|\sum_{{k: ~ \|k\|_\infty\le \|m\|_\infty/2}} 
 \langle k \rangle^{\varepsilon_1}  \Big|^q\Bigg)^{1/q}
 \\
&\ge  & C_1 \, 
\Bigg(\sum_{\|m\|_\infty \le N} \langle m \rangle^{(s_0 + \varepsilon_2 + \varepsilon_1 + n)q}\,  \Bigg)^{1/q} 
\\
&\ge  & C_2 \, N^{s_0 + \varepsilon_2 + \varepsilon_1 + n + n/q}
\eeqq
for some $C_1,C_2$ independent of $N$, see Substep 3.2 of the proof of Theorem \ref{mult2}.
The inequality \eqref{ws-311} yields 
\[
s_0 + \varepsilon_2 + \varepsilon_1 + n + n/q \le s_1 + \varepsilon_1 + n/q + s_2 + \varepsilon_2 + n/q
\]
which proves the claim.
\end{proof}

The duality argument used in the proof of Lemma \ref{notw} allows to treat the case $s_0 < 0$.

\begin{thm}\label{algebra20}
Let $1\leq p, p_1,p_2\leq \infty$ and $s_0, s_1, s_2 \in \re$.  
 Let $1/p \le (1/p_1) + (1/p_2)$,  $1 \le  q < \infty$,   $s_0 \le s_2 \le 0$, $0 \le s_1 + s_2$  and $s_1 + s_2-s_0>n/q$. 
There exists a constant $c$ such that 
\[ 
\|\, f \cdot g \, \|_{{M}_{p,q}^{s_0}} \leq c \, \| f \|_{{M}_{p_1,q'}^{s_1}}\,  \| g\|_{{M}_{p_2,q}^{s_2}} 
\]
holds for all $f\in {M}_{p_1,q'}^{s_1}$ and all $g \in {M}_{p_2,q}^{s_2}$.
\end{thm}

\begin{rem}
 \rm
Theorem \ref{algebra20} and Theorem \ref{mult8} have some overlap.
\end{rem}

\subsection*{Some further remarks to the literature}

Here we  recall  results of Iwabuchi \cite{iwabuchi} and Toft et all \cite{toftAlg}.
As Cordero and Nicola \cite{cordero} also Iwabuchi considered
the more general situation 
${M}_{p_1,q_1}^{s_1} \, \cdot \, {M}_{p_2,q_2}^{s_2} \hookrightarrow M^{s_0}_{p,q}$.
This greater flexibility with respect to the tripel $q,q_1,q_2$ allows to treat cases not covered by Theorems \ref{algebra12}, \ref{algebra20}.

\begin{prop}(Iwabuchi \cite{iwabuchi})\\
 Let $1\leq p, p_1,p_2 \leq \infty$, $1<q,q_1,q_2< \infty$ and $0 < s_0 < n/q$.
\\
{\rm (i)} If $q\ge q_1$, 
\be\label{ws-310}
\frac 1p \le \frac{1}{p_1} +  \frac{1}{p_2} \qquad \mbox{and}\qquad  1 + \frac 1q - \Big(\frac{1}{q_1} +  \frac{1}{q_2}\Big) = \frac {s_0}{n}\, , 
\ee
then there exists a constant $c$ such that
\[
\|\, f \cdot g \, \|_{{M}_{p,q}^{-s_0}} \leq c \, \| f \|_{{M}_{p_1,q_1}^{0}}\,  \| g\|_{{M}_{p_2,q_2}^{0}} 
\]
holds for all $f \in {M}_{p_1,q_1}^{0}$ and all $g \in {M}_{p_2,q_2}^{0}$.
\\
{\rm (ii)} Assume $q\ge \max (q_1,q_2)$ and \eqref{ws-310}. Then there exists a constant $c$ such that
\[
\|\, f \cdot g \, \|_{{M}_{p,q}^{s_0}} \leq c \, \| f \|_{{M}_{p_1,q_1}^{s_0}}\,  \| g\|_{{M}_{p_2,q_2}^{s_0}} 
\]  
holds for all $f \in {M}_{p_1,q_1}^{s_0}$ and all $g \in {M}_{p_2,q_2}^{s_0}$.
\end{prop}

\begin{rem}
 \rm
 Let us take $q=q_1 = q_2$. Then \eqref{ws-310} reads as $s_0=n/q'$. In combination with $0 <s_0<n/q$ this yields  $1< q< 2$.
Hence,  (i) reads as 
\[
\|\, f \cdot g \, \|_{{M}_{p,q}^{-n/q'}} \leq c \, \| f \|_{{M}_{p_1,q}^{0}}\,  \| g\|_{{M}_{p_2,q}^{0}} \, , 
\]
whereas (ii) gives
\[
\|\, f \cdot g \, \|_{{M}_{p,q}^{n/q'}} \leq c \, \| f \|_{{M}_{p_1,q_1}^{n/q'}}\,  \| g\|_{{M}_{p_2,q_2}^{n/q'}} \, .
\] 
\end{rem}

Toft et all \cite{toftAlg} also consider
the situation 
${M}_{p_1,q_1}^{s_1} \, \cdot \, {M}_{p_2,q_2}^{s_2} \hookrightarrow M^{s_0}_{p,q}$. Recall, $\mathring{M}^{s_0}_{p,q}$ denotes the closure of $\cs$ in 
$M^{s_0}_{p,q}$.

\begin{prop}\label{toft} (Toft et all \cite{toftAlg})
\\
Let $1\leq p,p_1,p_2,q,q_1,q_2\leq \infty$ and $s_0,s_1,s_2 \in \re$.
\\
{\rm (i)}
We suppose
\begin{itemize}
 \item[(a)] $1+ \frac 1p - \frac{1}{p_1} -  \frac{1}{p_2}\le 1 $;
 \item[(b)] $0 \le 1+ \frac 1q - \frac{1}{q_1} -  \frac{1}{q_2}\le 1/2$;
 \item[(c)] $s_0 \le \min (s_1,s_2)$;
 \item[(d)] $s_1 + s_2 \ge 0$;
 \item[(e)] $s_1 + s_2 -s_0 - n \, \Big(1+ \frac 1q - \frac{1}{q_1} -  \frac{1}{q_2}\Big) \ge 0$;
 \item[(f)] $s_1 + s_2 -s_0 - n \, \Big(1+ \frac 1q - \frac{1}{q_1} -  \frac{1}{q_2}\Big) > 0$ if 
$1+ \frac 1q - \frac{1}{q_1} -  \frac{1}{q_2}>0$ and either $s_1$ or $s_2$ or $-s_0$ equals 
$n \,\Big( 1+ \frac 1q - \frac{1}{q_1} -  \frac{1}{q_2}\Big)$.
\end{itemize}
Then  there exists a constant $c$ such that
\be\label{ws-320}
\|\, f \cdot g \, \|_{{M}_{p,q}^{s_0}} \leq c \, \| f \|_{{M}_{p_1,q_1}^{s_1}}\,  \| g\|_{{M}_{p_2,q_2}^{s_2}} 
\ee
holds for all $f \in \mathring{M}_{p_1,q_1}^{s_1}$ and all $g \in \mathring{M}_{p_2,q_2}^{s_2}$.
\\
{\rm (ii)} If \eqref{ws-320} holds for all $f,g \in \cs$, then (c), (d) and (e) follow.
\end{prop}

\begin{rem}
 \rm
 Again we consider the case $q=q_1=q_2$. Then (b) implies $1\le q \le 2$ and (e) reads as 
 $s_1 + s_2 -s_0 - n/q' \ge 0$.
 Hence, if we restrict us to  $1 < q \le  2$,  Proposition \ref{toft} is slightly  more general than 
Theorem \ref{algebra12} and Theorem \ref{algebra20}.
 However, for our purpose, see the next section on composition of functions,  Theorem \ref{algebra12} is already sufficient.
 Let us mention that Proposition \ref{subalgebra1} below, which is nothing but a modification of Lemma \ref{algebra10}, 
 is of central importance for the applications to composition operators we have in mind. 
\end{rem}

\subsection*{An important special case}

We consider $M^s_{2,2}$. A simple argument,  based on the frequency-uniform decomposition yields
$M^s_{2,2} = H^s$ in the sense of equivalent norms, see Remark \ref{hs=ms}.
For these Sobolev spaces $H^s$ almost all is known.
\begin{itemize}
 \item 
$H^s$ is an algebra with respect to pointwise multiplication if and only if $s>n/2$, 
see Strichartz \cite{St}, Triebel \cite[2.8]{triebel} or 
\cite[Thm.~4.6.4/1]{RS}. 
This coincides with Thm. \ref{mult2}.

\item 
Let $E$ be a Banach space of functions.  By $M(E)$ we denote the set of all pointwise multipliers of $E$, i.e., the set of all $f$
such that $T_f $, defined as $T_f (g) = f \cdot g$, maps $E$ into $E$.
We equip $M(E)$ with the norm $\| f\|_{M(E)}:= \|T_f\|_{\cl (E)}$.
For a description of $M(H^s)$ one needs the classes $H^{s,\ell oc}$. 
Here $H^{s,\ell oc}$ denotes the collection of all 
distributions $f\in \cs'$  such that $ f \, \cdot \, \varphi \in H^s$ for all  $\varphi \in C_0^\infty$.
In case $s>n/2$ it holds
\[
M(H^s) = \Big\{f \in H^{s,\ell oc}: \quad \|f\|_{M(H^s)}^*:= \sup_{\lambda \in \R} \| \psi (\, \cdot -\lambda)\, f\, \|_{H^s} < \infty\Big\}
\]
in the sense of equivalent norms.
Here $\psi$ is a smooth nontrivial cut-off function supported around the origin.
For all this we refer to Strichartz \cite{St}.
\item
In case $0 \le  s < n/2$ also characterizations of $M(H^s)$ are known, this time more complicated, based on capacities.
For all details we refer  to the monograph of Maz'ya and Shaposnikova \cite[Thm.~3.2.2, pp. 86]{MaSh2}.
\item
Now we concentrate on the situation described in Thm. \ref{algebra12} in case $0 < s < \frac n2$. As it is well-known, there
exists a constant $c$ such that 
\[ 
\|\, f \cdot g \, \|_{{H}^{2s-n/2}} \leq c \, \| f \|_{H^{s}}\,  \| g\|_{H^{s}} 
\]
holds for all $f, g \in {H}^{s}$, see, e.g.,  \cite[Thm.~4.5.2]{RS}. 
In Thm. \ref{algebra12} we proved that for any $\varepsilon >0$ there exists a constant $c_\varepsilon$ such that
\[ 
\|\, f \cdot g \, \|_{{M}^{2s-n/2-\varepsilon}_{1,2}} \leq c_\varepsilon \, \| f \|_{H^{s}}\,  \| g\|_{H^{s}} 
\]
holds for all $f, g \in {H}^{s}$.
We conjecture that ${M}^{2s-n/2-\varepsilon}_{1,2}$ and ${H}^{2s-n/2}$ are incomparable.
\end{itemize}


\section{Composition of functions} \label{comp}


There are some attempts to investigate composition of functions in the framework of modulation spaces, i.e.,
we consider the operator 
\be\label{tf}
T_f : \quad g \mapsto f \circ g, \qquad g \in M^s_{p,q}, 
\ee
and ask for mapping properties.
Of course, we used the symbol $T_f$ before with a different meaning, but we hope that will not cause problems.
Within Section \ref{comp} $T_f$ will have the meaning as in \eqref{tf}.
Based on pointwise multiplication one can treat  $f$ to be a polynomial
or even the more general case of $f$ being an entire function.


\subsection{Polynomials} \label{poly}


We consider the case

\[
f(z):= \sum_{\ell=1}^m \, a_\ell \, z^\ell\, , \qquad z \in \C\, ,
\]
where $m \in \N$, $m \ge 2$, and $a_\ell \in \C$, $\ell=1, \ldots \, , n$.
For brevity we denote the associated composition operator by $T_m$. 
In addition we need the abbreviation
\[
t_m (s):= s + (m-1)(s-n/q')\, , \qquad m=2,3,\ldots\, . 
\]

\begin{thm} \label{algebra13}
Let $1\leq p,q\leq \infty$ and $m \in \N$, $m \ge 2$.  
\\
{\rm (i)} Let either $s\ge 0$ and $q=1$ or $s>n/q'$. Then 
$T_m$ maps $M^s_{p,q}$ into itself. There exists a constant $c$ such that
\[
 \|\,  T_m g\, \|_{M^s_{p,q}} \le c \, \|\,  g\, \|_{M^s_{p,q}}\,  \sum_{\ell=1}^m \, |a_\ell| \,    \| \, g\, \|_{M^0_{\infty,1}}^{\ell-1}
\]
holds for all $g \in M^s_{p,q}$.
\\
{\rm (ii)}
Let $1 < q \le \infty$,   $0 <  s \le n/q'$ and $ t_{m}(s)> 0$. 
If $p \in [m,\infty]$ and $t<t_m(s)$,  then there exists a constant $c$ such that 
\[ 
\|\, T_m g \, \|_{{M}_{p/m,q}^{t}} \leq c \, \sum_{\ell=1}^m \, |a_\ell| \,   \| g\|_{{M}_{p,q}^{s}}^{\ell} 
\]
holds for all  $g \in {M}_{p,q}^{s}$.
\\
{\rm (iii)}
Let $q=1$ and    $s \ge 0$.  
If $p \in [m,\infty]$,  then there exists a constant $c$ such that 
\[ 
\|\, T_m g \, \|_{{M}_{p/m,1}^{s}} \leq c \, \sum_{\ell=1}^m \, |a_\ell| \,   \| g\|_{{M}_{p,1}^{s}}^{\ell} 
\]
holds for all  $g \in {M}_{p,1}^{s}$.
\end{thm}

\begin{proof}
{\em Step 1.}  
 Both  parts, (i) and (ii),  can be proved by induction based on Theorem \ref{mult2} or Theorem \ref{algebra12}.
 We concentrate on the proof of (ii).
Let $m=2$. 
Then by assumption $t_2 (s)= 2s-n/q' > 0 $. Hence, we may apply Theorem \ref{algebra12}
with $p_1=p_2 =p$ and $s_1=s_2$ and obtain
\[ 
\|\,  g^2 \, \|_{{M}_{p/2,q}^{t}} \leq c \,  \| g\|_{{M}_{p,q}^{s}}^2 
\]
for any $t< 2s-n/q'= t_2(s)$.
Now we assume that part (ii) is correct for all natural numbers in the interval $ [2,m]$. 
We split the product $g^{m+1}$ into the two factors $g^{m}$ and $g$.
By assumption $g^{m} \in {M}_{p/m,q}^{t}$ for any $t< t_{m} (s)$.
We put $s_1 = t = t_{m} (s) -\varepsilon$, $s_2 =s$, $p_1 = p/m$ and $p_2 = p$, where we assume that $\varepsilon >0$ is sufficiently small. 
This guarantees
\[
s_1 + s_2- \frac{n}{q'} =  s+ (m-1) \Big(s-\frac{n}{q'}\Big) -\varepsilon + s - \frac{n}{q'} = t_{m+1}(s)-\varepsilon >0 \,.  
\]
Hence, we may choose $s_0$ by
\[
s_0 < \min (s_1,s_2, t_{m+1}(s)-\varepsilon) = t_{m+1}(s)-\varepsilon\, . 
\]
Since $\varepsilon >0$ is arbitrary, any value $< t_{m+1} (s)$ becomes admissible for $s_0$.
An application of Theorem \ref{algebra12} yields
\[ 
\|\, g^{m} \cdot g \, \|_{{M}_{p/(m+1),q}^{s_0}} \leq c \, \| \, g_m \, \|_{{M}_{p/m,q}^{t_m-\varepsilon}}\,  \| g\|_{{M}_{p,q}^{s}} \, . 
\]
{\em Step 2.} Part (iii) is an immediate consequence of Lemma \ref{mult6}. 
\end{proof}

\begin{rem}
 \rm
For the case $s=0$ we refer to Cordero, Nicola \cite{cordero}, Toft \cite{toftCont} and Guo et all \cite{guo}.
\end{rem}


\subsection{Entire functions} \label{entire}


We consider the case of $f$ being an entire analytic function on $\C$, i.e.,
\[
f(z):= \sum_{\ell=0}^\infty \, a_\ell \, z^\ell\, , \qquad z \in \C\, ,
\]
where  $a_\ell \in \C$, $\ell=1, \ldots \, , n$.
Clearly, we need to assume $f(0)=a_0 = 0$. Otherwise $T_f g$ will not have global integrability properties. 
Let 
\[
f_0 (r):= \sum_{\ell=1}^\infty \, |a_\ell| \, r^\ell\, , \qquad r >0\, .
\]

\begin{thm} \label{algebra14}
Let $1\leq p,q\leq \infty$ and let either $s\ge 0$ and $q=1$ or $s>n/q'$. 
Let $f$ be an entire function satisfying $f(0)=0$.
Then 
$T_f$ maps $M^s_{p,q}$ into itself. There exist two constants $a,b$, independent of $f$, such that
\[
 \|\,  T_f g\, \|_{M^s_{p,q}} \le a \, f_0(b\, \|\,  g\, \|_{M^s_{p,q}})
\]
holds for all $g \in M^s_{p,q}$.
\end{thm}

\begin{proof}
The constant $c$ in Theorem \ref{algebra13}(i) depends on $m$.
To clarify the dependence on $m$ we proceed by induction.
Let $c_1$ be the best constant in the inequality 
\be\label{ws-40}
 \|\,  g_1\cdot  g_2\, \|_{M^s_{p,q}} \le c_1 \, (\|\,  g_1\, \|_{M^s_{p,q}}\,  \| \, g_2\, \|_{M^0_{\infty,1}}
 + \|\,  g_2\, \|_{M^s_{p,q}}\,  \| \, g_1\, \|_{M^0_{\infty,1}} ) \, ,
 \ee
see Lemma \ref{mult1}.
Further, let  $c_2$ be the best constant in the inequality 
\be\label{ws-41}
 \|\,  g_1\cdot  g_2\, \|_{M^0_{\infty,1}} \le c_2 \, \|\,  g_1\, \|_{M^0_{\infty,1}}\,  \| \, g_2\, \|_{M^0_{\infty,1}}\, ,
 \ee
see also Lemma \ref{mult1}.
 By $c_3$ we denote $\max(1,c_1,c_2)$. Our induction hypothesis consists in: the inequality 
 \[
 \|\, g^m\, \|_{M^s_{p,q}} \le c_3^{m-1} \, m  \, \|\,  g\, \|_{M^s_{p,q}} \| \, g\, \|_{M^0_{\infty,1}}^{m-1} 
\]
holds for all $ g \in {M^s_{p,q}}$ and all $m \ge 2$.
This follows easily from \eqref{ws-40} and \eqref{ws-41}.
Next we need the best constant, denoted by $c_4$, in the inequality
\[
\| \, g\, \|_{M^0_{\infty,1}} \le c_4\,  \| \, g\, \|_{M^s_{p,q}}\, , \qquad g \in M^s_{p,q}\, .
\]
This proves that
 \be\label{ws-42}
 \|\, g^m\, \|_{M^s_{p,q}} \le c_3^{m-1} \, m\, c_4^{m-1}  \, \|\,  g\, \|^m_{M^s_{p,q}} 
\ee
holds for all $ g \in {M^s_{p,q}}$ and all $m \ge 2$.
Hence
\beqq
\| \, T_f g\, \|_{M^s_{p,q}} & \le & \sum_{m=1}^\infty |a_m| c_3^{m-1} \, m\, c_4^{m-1}  \, \|\,  g\, \|^m_{M^s_{p,q}} \\
	& = & \frac{1}{c_3 \, c_4}\,  \sum_{m=1}^\infty |a_m|  \, m\, (c_3 \, c_4  \, \|\,  g\, \|_{M^s_{p,q}} )^m \, .
\eeqq
Since
\[
\sup_{m \in \N} m^{1/m} = 3^{1/3}
\]
the claimed estimate follows.
\end{proof}

\begin{rem}
 \rm
Theorem \ref{algebra14} is essentially known,  see, e.g., Sugimoto, Tomita and Wang \cite{Sugi} or Bhimani \cite{Bhim}. 
\end{rem}

\subsection*{One example}

The following example has been considered at various places.
Let $f(z):= e^z-1$, $z \in \C$. For appropriate constants $a,b>0$
it follows that
\be\label{ws-43}
 \|\,  e^g - 1\, \|_{M^s_{p,q}} \le a \, e^{b\, \|\,  g\, \|_{M^s_{p,q}}}
\ee
holds for all $g \in M^s_{p,q}$.
\\

It will be essential for our approach to non-analytic composition results that 
we can improve this estimate.


\subsection{Non-analytic superposition operators} \label{comp2}


There is a famous classical result by Katznelson \cite{Ka} (in the periodic case) and by 
Helson, Kahane,  Katznelson,  Rudin \cite{HKKR} (nonperiodic case) 
which says that only analytic functions operate on the Wiener algebra $\mathcal{A}$.
More exactly, the operator $T_f : ~ u \mapsto f(u)$ maps $\mathcal{A}$ into $\mathcal{A}$ if and only 
if $f(0)=0$ and $f$ is analytic. Here $\mathcal{A}$ is the collection of all $u \in C$ such that 
$\cf u \in L_1 $. Moreover, a similar result is obtained for particular standard modulation spaces. 
Bhimani and  Ratnakumar \cite{Bihami2}, see also Bhimani \cite{Bhim},  proved that  $T_f$ maps $M_{1,1}$ into $M_{1,1}$ if and only if $f(0)=0$ and $f$ is analytic. 
Therefore, the existence of non-analytic superposition results for weighted modulation spaces  is a priori not so clear.
\\
We shall concentrate on the algebra case. Our first aim consists in deriving a better estimate than \eqref{ws-43}. 
\\
To proceed we need some preparations.  An essential tool in proving our main result 
will be a certain subalgebra property.  Therefore we consider the 
following  decomposition of the phase space. Let $R>0$ and $\epsilon=(\epsilon_1,\ldots, \epsilon_n)$ be fixed with 
$\epsilon_j \in \{0,1\}$, $j=1,\ldots,n$. Then a decomposition of $\R$ into $(2^n+1)$ parts is given by
\[ P_R := \{ \xi\in\R :\: |\xi_j|\leq R, j=1,\ldots, n \} \]
and
\[ 
P_R(\epsilon) := \{ \xi\in \R: \: \sign (\xi_j) = (-1)^{\epsilon_j}, \: j=1,\ldots,n \} \setminus P_R. 
\]
For given $p,q,s$, $\epsilon=(\epsilon_1,\ldots, \epsilon_n)$ and $R>0$ we introduce the spaces
\[ 
{M}_{p,q}^{s}(\epsilon, R) := \{ f\in {M}_{p,q}^{s}: \: \supp \cf f \subset P_R(\epsilon) \}\, . 
\]

\begin{prop} \label{subalgebra1}
Let $1\leq p_1,p_2\leq \infty$, $1 < q \le \infty$ and $s_0 , s_1,s_2 \in \re$.  
Define $p$ by $\frac{1}{p}:= \frac{1}{p_1}+\frac{1}{p_2}$. Let $R > 2$.
If $p \in [1,\infty]$, $s_0 \le \min (s_1,s_2)$, $s_1, s_2 \ge 0$ and  $s_1 + s_2-s_0 > n/q'$, 
then there exists a constant $c$ such that 
\[ 
\|\, f \cdot g \, \|_{{M}_{p,q}^{s_0}} \leq c \,  (R-2)^{-[(s_1 +s_2 -s_0 ) - n/q']} \,  \| f \|_{{M}_{p_1,q}^{s_1}}\,  \| g\|_{{M}_{p_2,q}^{s_2}} 
\]
holds for all $f \in {M}_{p_1,q}^{s_1}(\epsilon,R)$ and all $g \in {M}_{p_2,q}^{s_2}(\epsilon,R)$.
The constant $c$ is independent from $R>2$ and $\epsilon$. 
\end{prop}

\begin{proof}
In order to show the subalgebra property we follow the same steps as in the
proof of Lemma \ref{algebra10}. We start with some almost trivial observations.
Let $f\in M_{p_1,q}^{s}(\epsilon, R)$ and $g\in M_{p_2,q}^{s}(\epsilon, R)$. By
\[ 
\supp (\cf f \ast \cf g) \subset \{ \xi +\eta: \xi \in \supp \cf f, \eta \in \supp \cf g \} 
\]
we have $\supp \cf (fg) \subset P_R(\epsilon)$. 
Let 
 \[
 P_R^* (\epsilon):= \Big\{k \in \Z: \quad \|k\|_\infty>R-1\, , \quad \sign (k_j) = (-1)^{\epsilon_j}, \: j=1,\ldots,n\Big\}.
 \]
 Hence, if $\supp \sigma_k \cap P_R(\epsilon) \neq \emptyset$, then $k \in P_R^*(\epsilon)$ follows.
Now we continue as in proof of Lemma \ref{algebra10}, Step 2, and obtain
\begin{eqnarray*}
 && \hspace{-0.9cm} \Big( \sum_{k\in P_R^* (\epsilon)}  \langle k\rangle^{s_0 q} \|\cfi \big(\sigma_k \cf (f\cdot g) \big)\|_{L_p}^q \Big)^{\frac{1}{q}} 
\\
& \leq & \sum_{\substack{t \in \Z, \\ -3<t_i < 3, \\ i=1,\ldots, n}} \Bigg( \sum_{k\in P_R^* (\epsilon)}  \langle k\rangle^{s_0q}
\Big[ \sum_{{l \in \Z: \atop t-l+k, l\in P_R^* (\epsilon)}} \| \cfi \big(\sigma_k \cf ( f_{t-(l-k)} \cdot g_l) \big)\|_{L_p} \Big]^q \Bigg)^{\frac{1}{q}} \, .
\end{eqnarray*}
This implies
\begin{eqnarray*}
&& \hspace{-0.9cm} \Big( \sum_{k\in P_R^* (\epsilon)}  \langle k\rangle^{s_0 q} \|\cfi \big(\sigma_k \cf (f\cdot g) \big)\|_{L_p}^q \Big)^{\frac{1}{q}} 
\\
& \leq & c_1\,  \sum_{\substack{t \in \Z, \\ 
-3<t_i < 3, \\ i=1,\ldots, n}} \Bigg( \sum_{k\in P_R^* (\epsilon)}  \langle k\rangle^{s_0q} \Big[ \sum_{{l \in \Z: \atop t-l+k, l\in P_R^* (\epsilon)}} 
\|f_{t-(l-k)}\cdot  g_l\|_{L_p} \Big]^q \Bigg)^{\frac{1}{q}}
\\
& \leq & c_2 \, \max_{\substack{t \in \Z, \\ -3<t_i < 3, \\ i=1,\ldots, n}} \Bigg( \sum_{k\in P_R^* (\epsilon)}  \langle k\rangle^{s_0q} 
\Big[ \sum_{{l \in \Z: \atop t-l+k, l\in P_R^* (\epsilon)}} \|f_{t-(l-k)}\|_{L_{p_1}} \|g_l\|_{L_{p_2}} \Big]^q \Bigg)^{\frac{1}{q}} \hspace{1cm}
\end{eqnarray*}
with $c_2$ and $ c_1$ as above. We put 
\begin{eqnarray*}
S_{1,t,k} & := & \hspace{-0.7cm} \sum_{\substack{l \in \Z: ~t-l+k, l\in P_R^* (\epsilon), \\
 |l|\leq |l-k|}}  \langle k-l\rangle^{s_1} \|f_{t-(l-k)}\|_{L_{p_1}}  \langle l\rangle^{s_2} \|g_l\|_{L_{p_2}} \, 
\langle k-l\rangle^{s_0 - s_1} \, \langle l\rangle^{-s_2} \, ; 
\\
S_{2,t,k} & := & \hspace{-0.7cm} \sum_{\substack{l \in \Z: ~ t-l+k, l\in P_R^* (\epsilon), \\ |l-k|\leq |l|}}  \langle k-l\rangle^{s_1} \, 
\|f_{t-(l-k)}\|_{L_{p_1}}  \langle l\rangle^{s_2} \, \|g_l\|_{L_{p_2}} \, \langle k-l\rangle^{- s_1} \, 
\langle l\rangle^{s_0-s_2} \, .
\end{eqnarray*}
H\"older's inequality leads to
\begin{eqnarray*}
S_{1,t,k} 
& \leq &   \Big( \sum_{\substack{j\in\Z, \\ |j+k|\leq |j|}} ( \langle j\rangle^{s_1} \|f_{t-j}\|_{L_{p_1}}  
\langle  j+k\rangle^{s_2} \|g_{j+k}\|_{L_{p_2}}   )^{q}\Big)^{1/q}  \\
			& & \qquad \qquad  \times \quad  \Big( \sum_{\substack{j\in\Z: t-j, j+k \in P_R^* (\epsilon)\\ 
|j+k|\leq |j|}} (\langle j \rangle^{s_0 - s_1} \, \langle j+k \rangle^{-s_2})^{q'}  \Big)^{\frac{1}{q'}} \, .
\end{eqnarray*}
Our assumptions $s_0 \le s_1$, $s_2 \ge 0$ and $s_1 + s_2 -s_0 >n/q'$ and $j+k \in P_R^* (\epsilon)$ imply
\beqq 
\hspace{-0.7cm} \Big( \sum_{\substack{j\in\Z, \\ 
|j+k|\leq |j|}} \Big|  && \hspace{-0.7cm}(\langle j \rangle^{s_0 - s_1} \, \langle j+k \rangle^{-s_2}) \Big|^{q'} \Big)^{\frac{1}{q'}} \le 
\Big( \sum_{m \in P_R^* (\epsilon)}  \langle m \rangle^{(s_0 -s_1 -s_2) q'} \Big)^{\frac{1}{q'}} 
\\
&\le & \Big( 2^{-n} \int_{\|x\|_\infty >R-2} (1+|x|^2)^{(s_0 -s_1 -s_2) q'/2}\, dx\Big)^{1/q'}
\\
& \le & \Big( 2^{-n} \int_{|x| > R-2} |x|^{(s_0 -s_1 -s_2) q'}\, dx\Big)^{1/q'}
\\
& \le & \Big(\frac{2^{-n}}{(s_1 + s_2 - s_0)q'-n}\Big)^{1/q'} \, (R-2)^{-[(s_1 +s_2 -s_0 ) - n/q']}  \, .
\eeqq
With $c_3 := \Big(\frac{2^{-n}}{(s_1 + s_2 - s_0)q'-n}\Big)^{1/q'} $ we insert this 
in our previous estimate and  obtain
\begin{eqnarray*}
\Big(\sum_{k \in \Z} S_{1,t,k}^q\Big)^{1/q} & \le  & c_3\, (R-2)^{-[(s_1 +s_2 -s_0 ) - n/q']}  \\
	& &  \hspace{-0.4cm} \times \,  \Bigg( \sum_{k\in\Z} \sum_{\substack{j\in\Z, \\ |j+k|\leq |j|}} \langle j \rangle^{s_1q}  \|f_{t-j}\|_{L^{p_1}}^q 
\langle j+k \rangle^{s_2q} \|g_{j+k}\|_{L^{p_2}}^q \Bigg)^{1/q}
\\
& \leq & c_3 \, c_4\,  (R-2)^{-[(s_1 +s_2 -s_0 ) - n/q']}  \,  \|g\|_{{M}^{s_2}_{p_2,q}} \|f\|_{{M}^{s_1}_{p_1,q}}
\, .
\end{eqnarray*}
Here $c_3,c_4$ are independent of $f,g, \epsilon$ and $R$.
For the second sum  the estimate
\[
\Big(\sum_{k \in \Z} S_{2,t,k}^q\Big)^{1/q} \le c_5\,  (R-2)^{-[(s_1 +s_2 -s_0 ) - n/q']} \, \|g\|_{{M}^{s_2}_{p_2,q}} \|f\|_{{M}^{s_1}_{p_1,q}} 
\]
follows by analogous computations. 
The proof is complete.
\end{proof}

Of course, the above arguments have a counterpart in case $q' = \infty$.

\begin{prop} \label{subal}
Let $1\leq p_1,p_2\leq \infty$, $q=1$ and $s_1,s_2 \in \re$.  
Define $p$ by $\frac{1}{p}:= \frac{1}{p_1}+\frac{1}{p_2}$. Let $R > 2$.
If $p \in [1,\infty]$,  $s_1, s_2 \ge 0$ and  $s_0 := \min (s_1, s_2)$, 
then there exists a constant $c$ such that 
\[ 
\|\, f \cdot g \, \|_{{M}_{p,1}^{s_0}} \leq c \,  (R-2)^{-(s_1 +s_2 -s_0 )} \,  \| f \|_{{M}_{p_1,1}^{s_1}}\,  \| g\|_{{M}_{p_2,1}^{s_2}} 
\]
holds for all $f \in {M}_{p_1,1}^{s_1}(\epsilon,R)$ and all $g \in {M}_{p_2,1}^{s_2}(\epsilon,R)$.
The constant $c$ is independent from $R>2$ and $\epsilon$. 
\end{prop}

As a consequence of Nikol'kij's inequality, see Lemma \ref{nikolskij},  Prop. \ref{subalgebra1} (with $s_0=s_1=s_2$ and $p_1=p$, $p_2 = \infty$) 
and Corollary \ref{einbettung} we obtain the following.

\begin{prop} \label{subalgebra2}
Let $1\leq p \leq \infty$ and  $R > 2$.
\\
{\rm (i)} Let $1 < q \le \infty$ and $ s>n/q'$.  
Then there exists a constant $c$ such that 
\[ 
\|\, f \cdot g \, \|_{{M}_{p,q}^{s}} \leq c \,  (R-2)^{-(s- n/q')} \,  \| f \|_{{M}_{p,q}^{s}}\,  \| g\|_{{M}_{p,q}^{s}} 
\]
holds for all $f, g \in {M}_{p,q}^{s}(\epsilon,R)$.
The constant $c$ is independent from $R>2$ and $\epsilon$. 
\\
{\rm (ii)} Let $q=1$ and $ s  \ge 0$.  
Then there exists a constant $c$ such that 
\[ 
\|\, f \cdot g \, \|_{{M}_{p,1}^{s}} \leq c \,  (R-2)^{-s} \,  \| f \|_{{M}_{p,1}^{s}}\,  \| g\|_{{M}_{p,1}^{s}} 
\]
holds for all $f, g \in {M}_{p,1}^{s}(\epsilon,R)$.
The constant $c$ is independent from $R>2$ and $\epsilon$. 
\end{prop}

Note that in the following we assume every function to be real-valued unless it is explicitly stated that 
complex-valued functions are allowed. To make this more clear we switch from $g \in M^s_{p,q}$ to $u \in M^s_{p,q}$. 
\\
Next we have to recall some assertions from harmonic analysis. 
The first one concerns a standard estimate of Fourier multipliers, see, e.g., \cite[Theorem 1.5.2]{triebel}.

\begin{lem} \label{bernstein}
Let $1 \le r \le \infty$ and assume that $s>n/2$. Then there exists a constant $c>0$ such that
\[ 
\|\cfi [ \phi \, \cf  g](\,\cdot \, )\|_{L_r} \leq c\,  \|\phi\|_{H^s} \, \|g\|_{L_r} 
\]
holds for all $g \in L_r$ and all $\phi \in H^s$.
\end{lem}

The next lemma is taken from \cite{brs}.

\begin{lem} \label{lma46brs}
Let $N\in \N$ and suppose $a_1, a_2,\ldots, a_N$ to be complex numbers. Then it holds
\[ 
a_1 \cdot a_2 \cdot \ldots \cdot a_N -1  = \sum_{l=1}^{N} \sum_{\substack{ j=(j_1,\ldots, j_l), \\ 
0\leq j_1 < \ldots < j_l \leq N }} (a_{j_1} -1) \cdot \ldots \cdot (a_{j_l}-1). 
\]
\end{lem}

In our approach the next estimate will be fundamental.

\begin{prop} \label{super}
Let $1 < p< \infty$, $1 \le q \le \infty$ and $ s>n/q'$.  
Then there exists a positive constant $C$ such that
\[
\|e^{i u} -1\|_{M_{p,q}^{s}}  \leq   C \, \|u\|_{M_{p,q}^{s}} \,  \left( 1 +  \| u\|_{M_{p,q}^s}\right)^{(s+n/q)(1+ \frac{1}{s-n/q'})}  
\]
holds for all real-valued $u \in M^s_{p,q}$.
\end{prop}

\begin{proof}
This proof follows  ideas developed  in \cite{brs}, but see also \cite{rrs}. 
\\
{\em Step 1.} Let $u$ be a nontrivial function in ${M}_{p,q}^s$ satisfying $\supp \cf u \subset P_R$ for some $R\geq 2$. \\
First  we consider the Taylor expansion
\[ 
e^{i u}-1 = \sum_{l=1}^r \frac{(i u)^l}{l!} + \sum_{l=r+1}^{\infty} \frac{(i u)^l}{l!} 
\]
resulting in the norm estimate
\[ 
\| e^{i u}-1\|_{{M}_{p,q}^s} \leq \Big\| \sum_{l=1}^r \frac{(i u)^l}{l!} 
\Big\|_{{M}_{p,q}^s} + \Big\| \sum_{l=r+1}^{\infty} \frac{(i u)^l}{l!} \Big\|_{{M}_{p,q}^s} \, . 
\]
For brevity we put 
\[
S_1:= \Big\| \sum_{l=1}^r \frac{(i u)^l}{l!} 
\Big\|_{{M}_{p,q}^s} \qquad \mbox{and}\qquad S_2 := \Big\| \sum_{l=r+1}^{\infty} \frac{(i u)^l}{l!} \Big\|_{{M}_{p,q}^s} \, .
\]
The natural number $r$ will be chosen later on.
Next we employ the algebra property, in particular the estimate \eqref{ws-42} with $C_1 := 2\, c_3 \, c_4$. We obtain
\[
S_2 \leq \sum_{l=r+1}^\infty \frac{1}{l!} \|u^l\|_{M_{p,q}^s} \leq \frac{1}{C_1} 
\sum_{l=r+1}^{\infty} \frac{ (C_1\, \|u\|_{M_{p,q}^s})^l}{l!}. 
\]
Now we choose $r$ as a function of $\|u\|_{M_{p,q}^s}$ and distinguish two cases:
\begin{enumerate}
\item $C_1\, \|u\|_{{M}_{p,q}^s}>1$. Assume that
\begin{equation}\label{ws-11b} 
3\, C_1\, \|u\|_{{M}_{p,q}^s} \leq r \leq 3\, C_1 \, \|u\|_{{M}_{p,q}^s} +1 
\end{equation}
and recall Stirling's formula $l! = \Gamma(l+1) \geq l^l e^{-l} \sqrt{2\pi l}$. Thus, we get
\begin{eqnarray*}
\sum_{l=r+1}^{\infty} \frac{(C_1\|u\|_{{M}_{p,q}^s})^l}{l!} & \leq &  
\sum_{l=r+1}^{\infty} \left( \frac{r}{l} \right)^l \left(\frac{e}{3}\right)^l \frac{1}{\sqrt{2\pi l}} 
\\
&\leq &   \sum_{l=r+1}^\infty \left(\frac{e}{3} \right)^l \leq  \frac{3}{3-e}.
\end{eqnarray*}
\item $C_1 \, \|u\|_{{M}_{p,q}^s}\leq 1$. It follows
\begin{equation*}
 \sum_{l=r+1}^{\infty} \frac{(C_1\, \|u\|_{{M}_{p,q}^s})^l}{l!} \leq  
C_1\, \|u\|_{{M}_{p,q}^s} \, \sum_{l=1}^{\infty} \frac{1}{l!}
\leq  C_1\,  e\,  \|u\|_{{M}_{p,q}^s}.
\end{equation*}
\end{enumerate}

Both together can be summarized as 
\begin{equation*} 
 S_2 \le C_2 \, \|u\|_{M_{p,q}^s}\, , \qquad C_2 := \max \Big(e, \frac{3}{C_1(3-e)}\Big).
\end{equation*}
To estimate $S_1$ we check the support of $\cf u^\ell $ and find
\begin{eqnarray*}
S_1 = \Big\| \sum_{l=1}^r \frac{(i u)^l}{l!} \Big\|_{M_{p,q}^{s}} & = & 
\Big( \sum_{k\in\Z} \langle k \rangle^{sq} \Big\|\Box_k \Big(\sum_{l=1}^{r} \frac{(i u)^l}{l!} \Big) 
\Big\|_{L_p}^q \Big)^{\frac{1}{q}} \\
& = & \Big( \sum_{\substack{k\in\Z, \\ -Rr-1<k_i< Rr+1, \\ i=1,\ldots,n}}  \langle k \rangle^{sq} 
\Big\|\Box_k \Big(\sum_{l=1}^{r} \frac{(i u)^l}{l!} \Big) \Big\|_{L_p}^q \Big)^{\frac{1}{q}} \\
& \leq & \Big( \sum_{\substack{k\in\Z, \\ -Rr-1<k_i< Rr+1, \\ i=1,\ldots,n}}  \langle k \rangle^{sq} 
\|\Box_k (e^{i u}-1) \|_{L_p}^q \Big)^{\frac{1}{q}} + S_2\, .
	\end{eqnarray*}
Concerning  $S_2$ we proceed as above. 
To estimate the first part we observe that
\[
C_3 := \sup_{k \in \Z} \, \| \, \sigma_k \, \|_{H^t} = \| \, \sigma_0 \, \|_{H^t} <\infty\, , 
\]
see Lemma \ref{bernstein}.
Furthermore,  
$\cos, \sin $ are Lipschitz continuous and consequently we get
\begin{eqnarray*}
\|\Box_k (e^{i u}-1) \|_{L_p} & \le &  C_3 \, \|e^{i u}-1 \|_{L_p} \\
& \le &   C_3 \, (\|\cos  u - \cos 0 \|_{L_p} + \|\sin  u - \sin 0 \|_{L_p})
\\
& \le &   2\, C_3 \, \| u - 0 \|_{L_p} \, .
\end{eqnarray*}
This implies
\begin{eqnarray*}
\Big( \sum_{\substack{k\in\Z, \\ -Rr-1<k_i< Rr+1, \\ i=1,\ldots,n}} \langle k \rangle^{sq} & & \hspace{-0.8cm}
\| \Box_k (e^{i u}-1) \|_{L_p}^q \Big)^{\frac{1}{q}} \\
& \le & 2\, C_3 \, \|\,  u \, \|_{L_p}\,  \Big( \sum_{\substack{k\in\Z, \\ -Rr-1<k_i< Rr+1, \\ i=1,\ldots,n}} \langle k \rangle^{sq}  
\Big)^{\frac{1}{q}} \, .
\end{eqnarray*}
Clearly, 
\begin{eqnarray*} 
\sum_{\substack{k\in\Z, \\ -Rr-1<k_i< Rr+1, \\ i=1,\ldots,n}} \langle k \rangle^{sq} 
& \le & \int_{\|\, x\, \|_\infty < Rr+1}\, \langle x \rangle^{sq}   \, dx
\nonumber
\\
& \le & \int_{|x|< \sqrt{n}(Rr+1)}\, \langle x \rangle^{sq}   \, dx
\nonumber
\\
& \le &    2\, \frac{\pi^{n/2}}{\Gamma (n/2)}\,  \int_{0}^{\sqrt{n} (Rr+1)} (1+\tau)^{n-1+sq}  \, d\tau
\nonumber
\\
& \le & 2\, \frac{\pi^{n/2}}{\Gamma (n/2)}\, \frac1{n+sq}\,  (\sqrt{n}(Rr+2))^{n+ sq}\,  \, . \hspace{1cm}
\end{eqnarray*}
To simplify notation we define
\[
C_4 := \Big(2\, \frac{\pi^{n/2}}{\Gamma (n/2)}\, \frac1{n+sq}\, \sqrt{n}^{n+sq}\Big)^{1/q}\, .
\]
In addition we shall use in case $1 < q \le \infty$
\[
\| u \|_{L^p}\le C_5 \, \|u\|_{{M}_{p,q}^{s}} \, , \qquad C_5 := \Big(\sum_{k\in \Z} \langle k \rangle^{-sq'} \Big)^{1/q'}
\]
which follows from H\"older's inequality and in case $q=1$
\[
\| u \|_{L^p}\le  \, \|u\|_{{M}_{p,1}^{s}} 
\]
as a consequence of triangle inequality.
Summarizing we have found
\[
\| e^{i u}-1\|_{{M}_{p,q}^s} \le \Big(2 \, C_2 + 2\, \max (C_5,1) \, C_4\,C_3\,  
(Rr+2)^{s+n/q}\Big)\,  \|u\|_{{M}_{p,q}^{s}} \, .
\]
Next we apply \eqref{ws-11b} which results in
\begin{equation} \label{expest}
\| e^{i u}-1 \|_{{M}_{p,q}^{s}} \leq C_6 \,  
\|u\|_{{M}_{p,q}^{s}} \left( 1 +  R \,  \| u\|_{M_{p,q}^s}\right)^{s+n/q} \, , 
\end{equation}
valid for all $u\in {M}_{p,q}^{s}$ satisfying $\supp \cf u \subset P_R$ and with positive constant $C_6$  not depending on $u$ and $R\ge 2$. 
\\
{\em Step 2.}	This time we consider  $u\in {M}_{p,q}^{s}$ without any restriction on the Fourier support. 
Here we need the restriction  $1< p< \infty$. For those $p$ the 
characteristic functions $\chi$ of cubes are Fourier multipliers in $L^p$ by the famous Riesz Theorem and therefore also in 
${M}_{p,q}^{s}$.
In addition we shall make use of the fact that the norm of the operator
$f \mapsto \cfi \chi\,  \cf f$ does not depend on the size of the cube.
Below we shall denote this norm by $C_7 = C_7 (p)$. 
We refer to Lizorkin \cite{Li} for all details.
For decomposing $u$ on the phase 
space we introduce functions $\chi_{R,\epsilon}$ and $\chi_R$, that is, the characteristic functions of the sets $P_R(\epsilon)$ and $P_R$, 
respectively. By defining
\begin{eqnarray*}
u_\epsilon (x) & = & \cfi [\chi_{R,\epsilon} (\xi) \, \cf u(\xi) ] (x) , \qquad x \in \R\, , 
\\
u_0 (x) & = & \cfi [\chi_R (\xi) \, \cf u (\xi) ] (x)\, , \qquad x \in \R\, , 
\end{eqnarray*}
we can rewrite $u$ as
\begin{equation} \label{urepr}
u(x) = u_0(x) + \sum_{\epsilon\in I} u_\epsilon (x) ,
\end{equation}
where $I$ is the set of all $\epsilon=(\epsilon_1,\ldots, \epsilon_n)$ with $\epsilon_j \in \{0,1\}$, $j=1,\ldots,n$. 
Hence
\begin{equation*} 
\| u\|_{{M}_{p,q}^{s}} \le  \| u_0\|_{{M}_{p,q}^{s}}  + 
\sum_{\epsilon\in I} \| u_\epsilon \|_{M_{p,q}^{s}}
\end{equation*}
and 
\[
\max \Big(\| u_0\|_{M_{p,q}^{s}},  \| u_\epsilon \|_{M_{p,q}^{s}}\Big) \le C_7 \, \| \, u\,  \|_{M_{p,q}^{s}}\, .
\] 
 Due to the  representation \eqref{urepr} and using an appropriate enumeration Lemma \ref{lma46brs} leads to
\[ 
e^{i u} -1 = \sum_{l=1}^{2^n +1} \sum_{0\leq j_1 <\ldots < j_l \leq 2^n} 
(e^{i u_{j_1}} -1)\cdot \ldots \cdot (e^{i u_{j_l}} -1) \, .
\]
The algebra property, in particular the estimate \eqref{ws-42} with $C_1 := 2\, c_3 \, c_4$, yields
\begin{equation}\label{ws-12}
\|e^{i u} -1\|_{M_{p,q}^{s}} \leq  \sum_{l=1}^{2^n +1} C^{l-1}_1\, 
\sum_{0\leq j_1 <\ldots < j_l \leq 2^n} \|e^{i u_{j_1}} -1 \|_{M_{p,q}^{s}} 
\cdot \ldots \cdot \| e^{i u_{j_l}} -1 \|_{M_{p,q}^{s}}. 
\end{equation}
By Proposition \ref{subalgebra2} and \eqref{expest} it follows
\begin{eqnarray}\label{ws-12b}
\| e^{i u_{j_k}} -1\|_{M_{p,q}^{s}} & = & \Big\|\sum_{l=1}^\infty \frac{(i u_{j_k})^l}{l!} \, \Big\|_{M_{p,q}^{s}}
\le \frac{R^{s-n/q'}}{c} \,  \Big(e^{c\, \|u_{j_k}\|_{M_{p,q}^{s}}/R^{s-n/q'} } -1 \Big) 
\nonumber 
\\
& \leq & \frac{(R-2)^{s-n/q'}}{c} \,  \Big(e^{c\, C_7 \, \|u\|_{M_{p,q}^{s}}/(R-2)^{s-n/q'} } -1 \Big) ,  
\end{eqnarray}
as well as
\begin{eqnarray}\label{ws-12c}
\| e^{i u_0}-1 \|_{{M}_{p,q}^{s}} \leq C_6 \, C_7\,   
\|u\|_{{M}_{p,q}^{s}} \left( 1 +  R \, C_7  \| u\|_{M_{p,q}^s}\right)^{s+n/q} \, ,
\end{eqnarray}
where we used the Fourier multiplier assertion mentioned at the beginning of this step.
The final step in our  proof is to choose the number $R$ 
as a function of $\|u\|_{M_{p,q}^{s}}$ such that \eqref{ws-12b} and \eqref{ws-12c} will be  approximately 
of the same size. 
\\
{\em Substep 2.1.} Let  $\|u\|_{M_{p,q}^{s}} \leq 1$. We choose $R=3$. 
Then \eqref{ws-12} combined with \eqref{ws-12b} and \eqref{ws-12c} results in the estimate 
\begin{equation*} 
 \|e^{i u} -1\|_{M_{p,q}^{s}} \leq C_8 \, \|u\|_{M_{p,q}^{s}} ,
\end{equation*}
where $C_8$ does not depend on $u$.\\
{\em Substep 2.2.} Let $\|u\|_{M_{p,q}^{s}} >1$. 
We choose $R\ge 3$ such that
\[ (R-2)^{s-n/q'} = \|u\|_{M_{p,q}^{s}} \, .\]
Now \eqref{ws-12},  combined with \eqref{ws-12b} and \eqref{ws-12c}, results in  
\begin{eqnarray}\label{ws-14}
 \|e^{i u} -1\|_{M_{p,q}^{s}} & \leq &  C_{9} \, \|u\|_{M_{p,q}^{s}} \,  \left( 1 +  \| u\|_{M_{p,q}^s}\right)^{(s+n/q)(1+ \frac{1}{s-n/q'})} \, , 
\,
\end{eqnarray}
with a constant $C_{9}$ independent of $u$.
\end{proof}

\begin{rem}
\rm
The restriction of $p$ to the interval $(1,\infty)$ is caused by  our decomposition technique, see Step 2 of the preceding proof.
We do not know whether Proposition \ref{super} extends to $p=1$ and/or $p = \infty$.
\end{rem}

Next, we need again a technical lemma.

\begin{lem} \label{ContExp}
Let $1 < p< \infty$, $1 \le q \le \infty$ and $ s>n/q'$.  
\\
{\rm (i)} The mapping $u \mapsto e^{iu}-1$
is locally Lipschitz continuous (considered as a mapping of $M_{p,q}^s$
into $M_{p,q}^s$). 
\\
{\rm (ii)} Assume $u\in M_{p,q}^s$ to be fixed and define a function 
$g: \re \mapsto M_{p,q}^s$ by $g(\xi) = e^{i u(x) \xi}-1$. Then the function $g$ is continuous.
\end{lem}

\begin{proof}
Local Lipschitz continuity follows from the identity
\begin{equation}\label{ws-15}
e^{iu}- e^{iv} = (e^{iv}-1)\, (e^{i(u-v)}-1) + (e^{i(u-v)}-1)\, ,
\end{equation}
the algebra property of $M_{p,q}^s$ and Proposition \ref{super}.
\\
To prove the continuity of $g$ we also employ the identity \eqref{ws-15}.
The claim follows by using the algebra property and Proposition \ref{super}.
\end{proof}
	
Now we are in position to prove the main result of this section.

\begin{thm} \label{superposition}
Let   $1< p < \infty$, $1\le q \le  \infty$ and $s>n/q'$. 
Let  $\mu$ be a complex measure on $\re$ such that
\begin{equation} \label{FourierEst}
L := \int_{-\infty}^\infty \, (1+|\xi|)^{1+(s+n/q)(1+ \frac{1}{s-n/q'})} \, d|\mu| (\xi) < \infty
\end{equation}
and such that $\mu(\re) = 0$. 
Furthermore, assume that the function $f$ is the inverse Fourier transform of $\mu$. 
Then $f$ is a continuous function and the composition operator $T_f: u \mapsto f \circ u$ maps $M_{p,q}^s$ into $M_{p,q}^s$.
\end{thm}

\begin{proof}
Equation \eqref{FourierEst} yields $\int_{\R} d|\mu|(\xi) < \infty$. Thus, $\mu$ is a finite measure and $\mu(\re) =0$ 
makes sense. Now we define the inverse Fourier transform of $\mu$
		\[ f(t) = \frac{1}{\sqrt{2\pi}} \int_{\R} e^{i \xi t} \, d \mu(\xi). \]
Moreover, since 
\[
(s+n/q)\Big(1+ \frac{1}{s-n/q'}\Big) >n
\]		
we conclude that $\int_{\re} |(i \xi)^j| \, d|\mu|(\xi) <\infty$, $j=1, \ldots\, , n+1 $, which implies 
$f \in C^{n+1}$. Due to $\mu(\re)=0$ we can also write $f$ as follows:
\[ 
f(t) = \frac{1}{\sqrt{2\pi}} \int_{\re} (e^{i \xi t}-1) \, d \mu(\xi). 
\]
Since $\mu$ is a complex measure we can split it up into real part $\mu_r$ and imaginary part $\mu_i$, 
where each of them is a signed measure. Without loss of generality we proceed our computations only with the positive 
real measure $\mu_r^+$. For all measurable sets $E$ we have $\mu_r^+(E) \leq |\mu|(E)$. 
\\
Let $u\in M_{p,q}^s$ and define the function $g(\xi) = e^{i u(x) \xi} -1$ analogously to 
Lemma \ref{ContExp}. Then $g$ is Bochner integrable because of its continuity and taking into account that the measure $\mu_r^+$ is finite. 
Therefore we obtain the Bochner integral
\[ 
\int_{-\infty}^\infty \big(e^{i u(x) \xi} -1 \big)\, d\mu_r^+(\xi) = \int_{-\infty}^\infty g(\xi) \, d\mu_r^+(\xi) 
\]
with values in $M_{p,q}^s$. By applying Minkowski inequality it follows
\[ 
\Big\| \int_{-\infty}^\infty \big(e^{i u(\cdot) \xi} -1 \big)\, d\mu_r^+(\xi) \Big\|_{M_{p,q}^s} 
\leq \int_{-\infty}^\infty \big\| e^{i u(\cdot) \xi} -1 \big\|_{M_{p,q}^s} \, d|\mu|(\xi). 
\]
Using the abbreviation $\| u\|:= \|u\|_{M_{p,q}^s}$,	Proposition \ref{super} together with equation \eqref{FourierEst} yields
\begin{eqnarray*}
\int_{|\xi| \|u\|\geq 1} && \hspace{-0.9cm} \big\| e^{i u(\cdot) \xi} -1 \big\|_{M_{p,q}^s} \, d|\mu|(\xi) 
\\
& \leq & 
C'\,  \| \, u\, \|_{M_{p,q}^s}^{1+(s+n/q)(1+ \frac{1}{s-n/q'})} \, \int_{|\xi| \|u\|\geq 1} | \xi|^{1+ (s+n/q)(1+ \frac{1}{s-n/q'})} \, d|\mu|(\xi) \\
				& < & \infty.
		\end{eqnarray*}
In a similar way also the remaining part  $|\xi| \le 1/ \|u\|$ of the integral can be treated. \\
The same estimates also hold for the measures $\mu_r^-$, $\mu_i^+$ and $\mu_i^-$. Thus, the result is obtained by
\begin{eqnarray*}
&& \hspace{-0.8cm} \|\sqrt{2\pi} f(u(x))\|_{M_{p,q}^s} 
\\
& = & 
\Big\| \int_{-\infty}^\infty  g(\xi) \, d\mu_r^+ - \int_{-\infty}^\infty  g(\xi) \, d\mu_r^-  + i \int_{-\infty}^\infty  g(\xi) \, d\mu_i^+ 
- i \int_{-\infty}^\infty  g(\xi) \, d\mu_i^- \Big\|_{M_{p,q}^s} 
\\
& \leq &  \int_{-\infty}^\infty  \| g(\xi) \|_{M_{p,q}^s} \, d|\mu_r^+|  + \int_{-\infty}^\infty  \| g(\xi) \|_{M_{p,q}^s}\, d|\mu_r^-| \\
& & \qquad + \int_{-\infty}^\infty  \| g(\xi)\|_{M_{p,q}^s} \, d|\mu_i^+| + \int_{-\infty}^\infty  \| g(\xi) \|_{M_{p,q}^s} \, d|\mu_i^-|
\, ,
\end{eqnarray*}
where every integral on the right-hand side is finite. Thus, the statement is proved.
\end{proof}

A bit more transparent sufficient conditions can be obtained by using Szasz theorem, see Peetre \cite[pp.9-11]{peetre} and \cite[Prop.~1.7.5]{ST}.
By $B^{s}_{p,q}(\re)$ we denote the Besov spaces on $\re$, see , e.g., \cite{triebel} or \cite{RS} for details.

\begin{lem}\label{szasz}
Let $t\ge 0$ and suppose  $f \in B^{t+1/2}_{2,1}(\re)$.
Then the Fourier transform of $f$ is a regular distribution
and
\[
\int_{-\infty}^\infty (1+|\xi|^2)^{t/2} |\cf f (\xi)|\, d\xi \le c\, \| f\|_{B^{t+1/2}_{2,1}(\re)}
\]
follows with some $c$ independent of $f$. 
\end{lem}

Based on Lemma \ref{szasz} and Theorem \ref{superposition} one obtains the next result.

\begin{cor}\label{szasz2}
Let   $1< p < \infty$, $1\le q \le  \infty$ and $s>n/q'$. 
Let   $f \in B^{t}_{2,1}(\re)$ for some 
\[
 t\ge \frac 32 + (s+n/q)\, \Big(1+ \frac{1}{s-n/q'}\Big)
\]
and suppose $f(0)=0$.
Then the composition operator $T_f: u \mapsto f \circ u$ maps real-valued functions in $M_{p,q}^s$ boundedly into $M_{p,q}^s$.
\end{cor}

\begin{proof}
Boundedness of $T_f$ follows  from Proposition \ref{super}, the proof of Theorem \ref{superposition} and Lemma \ref{szasz}. 
\end{proof}

\begin{rem}
 \rm
Let $t>0$ be given. 
A function $f:~ \re \to \re$, $m$-times continuously differentiable, compactly supported and  satisfying $f^m \in {\rm Lip}\, \alpha$ for some $\alpha \in (0,1]$,   belongs to  $ B^{t}_{2,1}(\re)$ if  $t< m+\alpha$.
\end{rem}

\subsection*{One example}

Ruzhansky, Sugimoto and Wang \cite{ruzhansky} suggested  to study the operator $T_\alpha$ associated to $f_\alpha (t):= t\, |t|^\alpha$, $t\in \re$, with $\alpha >0$.
This function belongs locally to the Besov space $B^{\alpha+1+1/p}_{p,\infty}(\re)$, $1 \le p \le \infty$, see \cite[Lemma~2.3.1/1]{RS} for a related case.
Let $\psi \in C_0^\infty (\re)$ be a smooth cut-off function such that $\psi (x)=1$ if $|x|\le 1$.
Then the function
\[
\tilde{f}_{\alpha, \lambda} (t) := \psi (t/\lambda) \cdot f_\alpha (t), \qquad t \in \re \, , 
\]
belongs to $B^{\alpha+1+1/p}_{p,\infty}$ for any $p$,  $1 \le p \le \infty$, and any $\lambda >0$.  
Applying Corollary \ref{szasz2} and
\[
u(x)\, |u(x)|^\alpha = \tilde{f}_{\alpha, \lambda} (u(x))\, , \qquad x \in \R\, , \quad \lambda := \|u\|_{L_\infty}\, , 
\]
we find the following.

\begin{cor} \label{beispiel}
Let   $1< p < \infty$, $1\le q \le  \infty$ and $s>n/q'$. 
Let $\alpha$ be a positive real number such that 
\[
(s+n/q)\Big(1+ \frac{1}{s-n/q'}\Big) < \alpha \, .
\]
Then the composition operator $T_\alpha :~ u \mapsto u \, |u|^\alpha $ maps real-valued functions in $M_{p,q}^s$ boundedly into $M_{p,q}^s$.
\end{cor}

\subsection*{The special case $p=q=2$}

Finally we will have a look onto the special case $M^s_{2,2} = H^s$, $s>n/2$.
In Bourdaud, Moussai, S. \cite{bms} the set of functions $f$ such that $T_f: ~ g \mapsto f \circ g $ maps $H^s$ into itself
has been characterized.

\begin{prop}\label{compose}
Let  $s> \frac 12 \,  \max (n,3)$.
For a Borel measurable function $f:\, \re \to \re$ the
composition operator $T_f$ acts  on
 $H^{s}$  if and only if $f(0)=0$ and $f\in H^{s,\ell oc} (\re)$.
\end{prop}

Concerning our example $T_\alpha $ treated above this yields the following:
$T_\alpha$ maps $H^s$ into itself if and only if 
$\alpha > s-3/2$ (instead of $\alpha > s+\frac n2 + \frac{s+ n/2}{s-n/2}$ as required in Corollary \ref{beispiel}).
\\
Corollary \ref{szasz2} and Corollary \ref{beispiel} may be understood as first results about sufficient conditions, not more.
\\

\subsection*{A final remark}

The method employed here  has been used before in connection with composition operators on  Gevrey-modulation spaces 
and modulation spaces of ultradifferentiable functions, see Bourdaud \cite{Bou}, Bourdaud, Reissig, S. \cite{brs}, 
Reich, Reissig, S. \cite{brs}, and Reich \cite{reich2}, for H\"ormander-type spaces $B_{p,k}$ we refer to Jornet, Oliaro \cite{Joroli}. 
It would be desirable to develop this method more systematically.

\end{document}